\documentclass[a4paper,twoside,12pt]{amsart}      
\usepackage{amsmath,amssymb}
\usepackage{amsthm}
\usepackage{graphics}                 
\usepackage{color}                    
\usepackage{hyperref}                 
\usepackage{graphicx,eucal}
\usepackage[all]{xy}
\usepackage{soul}
\usepackage{geometry}
\usepackage{changebar}

\newtheorem{theorem}{Theorem}[section]
\newtheorem{proposition}[theorem]{Proposition}
\newtheorem{corollary}[theorem]{Corollary}
\newtheorem{lemma}[theorem]{Lemma}

\theoremstyle{definition}

\newtheorem{remark}[theorem]{Remark}

\newcommand{\C}{\mathbb{C}}
\newcommand{\bbP}{\mathbb{P}}
\newcommand{\OO}{\mathcal{O}}

\newcommand{\Fiso}{\mathcal{F}_{\textrm{iso}}}

\newcommand{\rh}{\mathrm{Mon}}
\newcommand{\systems}{\mathfrak{sl}_2(S)}

\newcommand{\SL}{\mathrm{SL}_2}
\newcommand{\sli}{\mathfrak{sl}_2}
\newcommand{\Nu}{\mathcal V}

\newcommand{\Rep}{\Xi}
\newcommand{\SYS}{\mathrm{Syst}}
\newcommand{\PSYS}{\mathrm{Syst}^{\mathbb P^1}}
\newcommand{\CON}{\mathrm{Con}}
\newcommand{\PCON}{\mathrm{Con}^{\mathbb P^1}}

\newcommand{\CONSSS}{\Upsilon}
\newcommand{\PCONSSS}{\Upsilon^{\mathbb P^1}}
\newcommand{\Teich}{\mathrm{Teich}(S)}
\newcommand{\Mod}{\mathrm{Mod}(S)}

\newcommand{\curvename}{X}
\newcommand{\curvenametop}{S}
\newcommand{\fundgroup}{\pi_1}
\newcommand{\diff}{\mathrm{d}}
\newcommand{\oneforms}{\Omega^1}
\newcommand{\hypcover}{h}
\newcommand{\Darbouxcover}{\Psi}

\date{\today}
\title[The Riemann-Hilbert mapping for systems]{The Riemann-Hilbert mapping for $\mathfrak{sl}_2$-systems over genus two curves}

\author[G. Calsamiglia]{Gabriel Calsamiglia}
\address{Instituto de Matem\'atica e Estat\' \i stica, Universidade Federal Fluminense,
Rua Professor Marcos Waldemar de Freitas Reis, s/n, Bloco H - Campus do Gragoat\' a
S\~ ao Domingos- Niter\' oi - RJ - CEP: 24.210-201}
\email{gabriel@mat.uff.br}

\author{Bertrand Deroin}
\address{Laboratoire AGM - CNRS / Université Cergy-Pontoise - 2 av. Adolphe Chauvin - 95302 Cergy-Pontoise, France}
\email{bertrand.deroin@u-cergy.fr}

\author[V. Heu]{Viktoria Heu}
\address{IRMA, 7 rue Ren\'e-Descartes, 67084 Strasbourg Cedex, France}
\email{heu@math.unistra.fr}

\author[F. Loray]{Frank Loray}
\address{Univ Rennes, CNRS, IRMAR - UMR 6625, F-35000 Rennes, France}
\email{frank.loray@univ-rennes1.fr}

\dedicatory{\`a \'Etienne Ghys}
\thanks{}

\begin{document}
\begin{abstract}
We prove in two different ways that the monodromy map from the space of irreducible
$\sli$-differential-systems on genus two Riemann surfaces, towards the character variety of
$\SL$-representations of the fundamental group, is a local diffeomorphism.
We also show that this is no longer true in the higher genus case.
Our work is motivated by a question raised by \'Etienne Ghys about Margulis' problem: the existence
of curves of negative Euler characteristic  in compact quotients of $\SL(\C)$.
\newline
Nous montrons de deux mani\`eres diff\'erentes que l'application monodromie, depuis l'espace
des $\sli$-syst\`emes diff\'erentiels irr\'eductibles sur les surfaces de Riemann de genre deux,
vers la vari\'et\'e de caract\`eres des $\SL$-repr\'esentations du groupe fondamental,
est un diff\'eomorphisme local. Nous montrons aussi que ce n'est plus le cas en
genre sup\'erieur. Notre travail est motiv\'e par une question d'\'Etienne Ghys \`a propos
d'un probl\`eme de Margulis : l'existence de courbes de caract\'eristique d'Euler n\'egative
dans les quotients compacts de $\SL(\C)$.

\end{abstract}
  \thanks{The authors would like to warmly thank the referee for the careful reading and helpful suggestions on the manuscript.
This work was supported by ANR-13-BS01-0001-01,
ANR-13-JS01-0002-01, ANR-16-CE40-0008, Math-Amsud Program, Capes 88887.159532/2017-00, CNPq  and the France-Brazil agreement in Mathematics}
\maketitle
\section{Introduction}

Let $S$ be a compact oriented topological surface of genus $g\geq 2$ and $\curvename\in\Teich$ a complex structure on $S$, \emph{i.e.} $X$ is a smooth projective curve endowed with the isotopy class of a diffeomorphism $X\to S$.
Given a $\sli$-matrix of holomorphic one forms on $\curvename $:
$$A=\begin{pmatrix}\alpha&\beta\\ \gamma&-\alpha\end{pmatrix}\in\sli(\oneforms(\curvename ))$$
we consider the system of differential equations
for $Y\in\C^2$:
\begin{equation} \label{eq:linear system} \diff Y+A Y=0.\end{equation}
A fundamental matrix $B(x)$ at a point $x_0\in \curvename $ is a two-by-two matrix $B\in\SL(\mathcal O_{x_0})$ whose columns form
a base for the two-dimensional vector space of solutions, i.e. satisfying  $\diff B+AB=0$, $\det(B)\equiv 1$.
It can be analytically continued as a function $B: \widetilde{X} \rightarrow \text{SL}_2(\mathbb  C) $  defined on the universal cover of $X$ 
which satisfies an equivariance
$$\forall \gamma\in\fundgroup (\curvename ),\ \ \ B(\gamma\cdot x)=B(x)\cdot \rho_A(\gamma)^{-1}$$
for a certain representation $\rho_A : \fundgroup (\curvenametop ) \rightarrow \text{SL}_2(\mathbb{C})$.
The conjugacy class of $\rho_A$ in the $\text{SL}_2(\mathbb{C})$-character variety 
$$\Rep:=\text{Hom}(\fundgroup (\curvenametop ),\text{SL}_2(\mathbb{C}))//\text{SL}_2(\mathbb{C})$$
does not depend on the initial solution and will be referred to as the {\bf monodromy class} of the system.
Also, for any $M\in\text{SL}_2(\mathbb{C})$, the monodromy class of $MAM^{-1}\in\mathfrak{sl}_2(\oneforms(\curvename ))$
coincides with that of $A$. It is therefore natural to consider the space of systems up to gauge equivalence:
$$\SYS:=\{(\curvename ,A):\curvename \in\Teich, A\in\mathfrak{sl}_2(\oneforms(\curvename ))\}//\text{SL}_2(\mathbb{C}).$$

The Riemann-Hilbert mapping is the map
$$\rh:\SYS\rightarrow\Rep$$
defined by $\rh(\curvename ,[A]):=[\rho_A]$. Both $\SYS$ and $\Rep$ are (singular) algebraic varieties of  complex dimension $6g-6$.
The irreducible locus $\SYS^{\text{irr}}\subset \SYS$ and $\Rep^{\text{irr}}\subset \Rep$, characterized by those
$A$ and $\mathrm{image}(\rho)$ without non trivial  invariant subspace, define smooth open subsets.
Then, $\text{Mon}$ induces a holomorphic mapping between these open sets.
Our main aim is to prove

\begin{theorem}\label{t:localdiffeo} If $S$ has genus two, then the holomorphic map
$$\rh:\SYS^{\text{irr}}\rightarrow\Rep^{\text{irr}}$$
is a local diffeomorphism.
\end{theorem}

The equivalent statement is not true in general for higher genera. Easy counterexamples in genus at least $4$  can be constructed by considering the pull back of a system on a genus two  Riemann surface $X$ by a parametrized family of ramified coverings over $X$. In genus $g=3$ there are also counterexamples (see Section \ref{s:rational curves correspond to systems} for details).
On the other hand,  we note that irreducibility
is a necessary assumption. Indeed, diagonal and nilpotent systems
$$A=\begin{pmatrix}\alpha&0\\0&-\alpha\end{pmatrix}\ \ \ \text{and}\ \ \ \begin{pmatrix}0&\alpha\\0&0\end{pmatrix}$$
admit non trivial isomonodromic deformations: this comes from isoperiodic deformations of pairs $(\curvename ,\alpha)$
that exist, as it can be seen just by counting dimensions (see also  \cite{McMullen}).

\vskip 0.5cm

{\bf Motivation.}
The question of determining the properties of the monodromy representations associated to  holomorphic $\mathfrak{sl}_2$-systems on a Riemann surface $\curvename $ of genus $g>1$ was raised by Ghys.
The motivation comes from the study of quotients $M:=\SL\slash\Gamma$ by
cocompact lattices $\Gamma\subset\SL$. These compact complex manifolds are not K\"ahler.
Huckleberry and Margulis proved in \cite{HuckleberryMargulis} that they admit no complex hypersurfaces
(and therefore no non-constant meromorphic functions).
Elliptic curves exist in such quotients, while the existence of compact curves of genus at least two remains open and is related to Ghys' question. Indeed,
assuming that for a non trivial system on a curve $\curvename$, its monodromy has image contained in $\Gamma$ (up to conjugation), then the corresponding fundamental matrix induces a non trivial holomorphic map from $\curvename$ to $M$. Reciprocally, any curve $\curvename$ in $M$ can be lifted to $\text{SL}(2,\mathbb C)$ and gives rise to the fundamental matrix of some system on $\curvename$, whose monodromy is contained in $\Gamma$. In fact, it is not known whether holomorphic $\mathfrak{sl}_2$-systems on Riemann surfaces of genus $>1$ give rise to representations with discrete or real image. Although Ghys' question remains open, our result shows { on the one hand that we can locally realize arbitrary deformations of the monodromy representation of a given system over a genus two curve by allowing deformations of both the curve and the system, and on the other that, if a genus two curve exists in some $M$ as before, it is rigid in $M$ up to left translations. }

\vskip 0.5cm

{\bf Idea of the proofs.} We propose two different proofs of our result, using isomonodromic deformations of two kinds of objects,
namely  vector bundles with connections
used by the last two authors, and branched projective structures used by the first two authors. Both proofs were obtained independently. We decided to write them together in this paper.
{
We first develop the approach with flat vector bundles in Sections \ref{sec:systems}, \ref{Sec:FuchsianSystems} and \ref{Sec:Proof1}. It is based on the work \cite{HL} by the last two authors, where the arguments occur at the level of (a finite covering of) the moduli space of systems. In more detail, one considers a system as a holomorphic $\sli$-connection
$\nabla=\diff +A$ on a trivial bundle $\curvename\times\mathbb{C}^2\to \curvename $,
and thinks of it as a point in the (larger) moduli space  $\CON$ of all triples $(\curvename ,E,\nabla)$ where $E\to \curvename$ is a holomorphic rank two vector bundle and $\nabla$ is a flat $\sli$-connection on $E$.
The subspace $\SYS$ of those triples over trivial bundles has codimension $3$.
The monodromy map is locally defined on the larger space $\CON$ and its level sets induce  a singular foliation $\Fiso$ by $3$-dimensional leaves:
the isomonodromy leaves. We note that everything is smooth in restriction to the irreducible locus, and $\rh$
is a submersion. Since $\Fiso$ and $\SYS$ have complementary dimensions,
the fact that the restriction of $\rh$ to $\SYS^{\text{irr}}$ is a local diffeomorphism is equivalent to
the transversality of $\Fiso$ to $\SYS^{\text{irr}}$. To prove this, we strongly use the hyperellipticity of genus $2$ curves
to translate our problem to some moduli space of logarithmic connections on $\bbP^1$.
The main tool here, due to Goldman, is that irreducible $\text{SL}_2(\mathbb{C})$-representations
are invariant under the hyperelliptic involution $h:X\to X$, and descend to representations of the orbifold
quotient $X/ h$.}
There, isomonodromy equations
are well-known, explicitely given by a Garnier system, and we can compute the transversality.

The approach using branched projective structures is developed in Sections \ref{s:BPS}, \ref{s:rational curves correspond to systems}, \ref{s:the tangent bundle}, \ref{s:rigidity} and \ref{s:proof}. We hope that it might be generalized to higher genus. It uses isomonodromic deformation spaces of branched complex projective
structures over a surface $S$ of genus $g\geq 2$, that were introduced in \cite{CDF}; namely given a conjugacy class of irreducible representation $\rho: \pi_1(S) \rightarrow \text{SL}(2,\mathbb C)$, and an even integer $k$, the space of complex projective structures with $k$ branch points (counted with multiplicity) and holonomy $\rho$ has the structure of a smooth $k$-dimensional complex manifold denoted by $\mathcal M_{k,\rho}$. We establish a dictionary between
\begin{enumerate}
\item[a)] Systems on $X\in\Teich$ with monodromy $\rho$
\item[b)] Regular holomorphic foliations on $X\times\mathbb{P}^1$ transverse to the $\mathbb{P}^1$-fibration and with monodromy $[\rho]\in \text{Hom}(\pi_1(S),\text{PSL}_2(\mathbb{C}))$ (Riccati foliations)
\item[c)] Complete rational curves in the space $\mathcal{M}_{2g-2,\rho}$ of branched projective structures over $S$ with branching divisor of degree $2g-2$ and monodromy $\rho$.
\end{enumerate}
The equivalence between (a) and (b) is an easy and well-known fact. The interesting aspect of the dictionary is between (a) and (c), it is discussed in Section~\ref{s:rational curves correspond to systems}. Injectivity of the differential of the Riemann-Hilbert mapping at a given $\mathfrak{sl}(2)$-system with monodromy $\rho$ is then equivalent to first order rigidity of the corresponding rational curve in $\mathcal M_{2g-2,\rho}$. In the genus two case, the moduli space $\mathcal M_{2,\rho}$ is a complex surface, and the rigidity of the rational curve is established by showing its self-intersection is equal to $-4$. In higher genus, the infinitesimal rigidity of the rational curve does not hold, counter-examples are described in Section \ref{r:counter-example}.

{
Finally, in Section \ref{Sec:Comparaison}, we compare the objects involved in the two proofs.
A branched projective structures on $X$ can be viewed as triple $(P,\mathcal F,\sigma)$ where
\begin{itemize}
\item $P\to X$ is a ruled surface (i.e. total space of a $\mathbb P^1$-bundle),
\item $\mathcal F$ is a regular Riccati foliation on $P$ (i.e. transversal to all $\mathbb P^1$-fibers),
\item $\sigma:X\to P$ is a section which is not $\mathcal F$-invariant.
\end{itemize}
Branch points come from tangencies between the section $\sigma(X)$ and the foliation $\mathcal F$.
On the other hand, a $\text{SL}(2,\mathbb C)$-connection $(E,\nabla)$ defines, after projectivization,
a projective connection $\mathbb P\nabla$ on the $\mathbb P^1$-bundle $\mathbb P E$ whose
horizontal sections are the leaves of a Riccati foliation $\mathcal F$ on the total space $P$ of the bundle.
In this setting, a section $\sigma:X\to P$ inducing a projective structure with $2g-2$ branch points
corresponds to a line sub-bundle $L\subset E$ having degree $0$, implying in particular that $E$ is
strictly semi-stable. We use this dictionary in Section \ref{Sec:Comparaison} to explain how isomonodromic
deformations considered in each proof are related. This allows us to have a better understanding of the geometry of isomonodromic deformations of $\text{SL}(2,\mathbb C)$-connections around points in the locus $\SYS$, where $E$ is a trivial bundle.
}

\section{$\sli$-systems on $\curvename $}\label{sec:systems}

Let us consider the smooth projective curve of genus two defined in an affine chart by
\begin{equation}\label{DefCurve}\curvename_{\boldsymbol{t}}:=\{y^2=x(x-1)(x-t_1)(x-t_2)(x-t_3)\}\end{equation}
for some parameter
\begin{equation}\label{defTnotTeich}\boldsymbol{t}=(t_1,t_2,t_3)\in T:=\left(\bbP^1\setminus\{0,1,\infty\}\right)^3\setminus\bigcup_{i\not=j}\{t_i=t_j\}.\end{equation}
{
  The space $T$ is a finite covering of the moduli space of genus two curves.
A $\sli$-system on $ \curvename_{\boldsymbol{t}}$ takes the form
\begin{equation}\label{DefSysA}\mathrm{d}Y+A Y=0\ \ \ \text{with}\ \ \ A=\begin{pmatrix} \alpha&\beta\\ \gamma&-\alpha\end{pmatrix}
\ \ \ \text{and}\ \ \ Y=\begin{pmatrix} y_1\\ y_2\end{pmatrix}\, , \end{equation}
}
where $\alpha, \beta, \gamma$ are holomorphic $1$-forms on $\curvename_{\boldsymbol{t}}$.
It can be seen as the equation $\nabla Y=0$ for $\nabla$-horizontal sections for the $\sli$-connection
$\nabla=\diff +A$
on the trivial bundle over $\curvename_{\boldsymbol{t}}$.

We denote by $\SYS(\curvename_{\boldsymbol{t}})$ the moduli space
of systems on $\curvename_{\boldsymbol{t} }$ modulo $\SL$-gauge action:
$$Y\mapsto MY\ \ \ \rightsquigarrow\ \ \ A\mapsto MA M^{-1},\ \ \ \text{for}\ M\in\SL.$$
An invariant is evidently given by the determinant map
$$\SYS(\curvename_{\boldsymbol{t}} )\to \mathrm{H}^0(\curvename_{\boldsymbol{t}} , \oneforms\otimes\oneforms)\ ;\
A\mapsto \det(A)=-(\alpha\otimes\alpha+\beta\otimes\gamma).$$
This map actually provides the categorical quotient for this action.
More precisely, we have (see \cite[section 3.3]{HL})

\begin{proposition}\label{PropReduSystem}The system $A$ in \eqref{DefSysA} is reducible if, and only if,
$\det(A)$ is a square, i.e. $\det(A)=\alpha'\otimes\alpha'$ for a $1$-form $\alpha'$.
Moreover, in the irreducible case, the conjugacy class of $A$ is determined by $\det(A)$.
\end{proposition}

More explicitely (see \cite[proof of prop. 3.3]{HL}), in the irreducible case, we can conjugate $A$ to a normal form
\begin{equation}\label{normalfA}A=\begin{pmatrix} 0&\beta\\ \gamma&0\end{pmatrix}=\begin{pmatrix} 0&(\beta_1x+\beta_0)\frac{\diff x}{y}\\ (\gamma_1x+\gamma_0)\frac{\diff x}{y}&0\end{pmatrix}\end{equation}
with $\beta_0,\beta_1,\gamma_0,\gamma_1\in \C$.
{
 Indeed, the hyperelliptic involution \(  (x, y) \mapsto (x,- y) \) on \( X_{\bf t}\)  acts trivially on the character variety for the group \( \text{SL} (2,\C) \) by a theorem of Goldman (see \cite{Goldman2}), hence it also fixes the system. This means that we can lift the hyperelliptic involution to an automorphism of the system.
After taking a convenient gauge transformation, the automorphism takes the form \[ (x, y , y_1,y_2) \in X_{\bf t}\times \C^2 \mapsto (x, - y, i y_1, -i y_2)\in X_{\bf t} \times \C ^2,\] which immediately yields that \( \alpha \) vanishes identically.
}
The normal form (\ref{normalfA}) is unique up to conjugacy by (anti-) diagonal matrices
$$M=\begin{pmatrix}\lambda&0\\0&\lambda^{-1}\end{pmatrix}\ \ \ \text{or}\ \ \ \begin{pmatrix}0&\lambda\\-\lambda^{-1}&0\end{pmatrix}$$
and the determinant is given by
$$\det(A)=\frac{-(\beta_1x+\beta_0)(\gamma_1x+\gamma_0)}{x(x-1)(x-t_1)(x-t_2)(x-t_3)}\diff x\otimes \diff x.$$
The system of the form \eqref{normalfA} is reducible if and only if $\beta_0\gamma_1-\beta_1\gamma_0=0$. Denote by $\boldsymbol\nu=(\nu_0,\nu_1,\nu_2)\in \C^3$
the variable of the space of quadratic differentials
$$\mathrm{H}^0(\curvename , \oneforms\otimes\oneforms)\ni\frac{\nu_2x^2+\nu_1x+\nu_0}{x(x-1)(x-t_1)(x-t_2)(x-t_3)}\diff x\otimes \diff x.$$ Then Proposition \ref{PropReduSystem} can be reformulated as follows.

\begin{corollary} The moduli space $\SYS^{\text{irr}}(\curvename_{\boldsymbol{t}} )$ of irreducible $\sli$-systems over the curve $\curvename_{\boldsymbol{t}}$  identifies with
$$\SYS^{\text{irr}}(\curvename_{\boldsymbol{t}})\stackrel{\sim}{\longrightarrow}  \Nu\ ;\ A\mapsto  \boldsymbol\nu=\det(A),$$
where
$\Nu=\C^3_{\boldsymbol\nu}\setminus\{\nu_1^2-4\nu_0\nu_2=0\}$.
\end{corollary}

Let now $\SYS^{\text{irr}}$ denote the family of moduli spaces $\SYS^{\text{irr}}(\curvename_{\boldsymbol{t}} )$, where the parameter $\boldsymbol{t}$ defining the curve varies in the space $T$  defined in \eqref{defTnotTeich} :
$$\SYS^{\text{irr}}:=\left\{ (\curvename_{\boldsymbol{t}} , A ) ~|~\boldsymbol{t}\in T \, , \, A\in \SYS^{\text{irr}}(\curvename_{\boldsymbol{t}})\right\}.$$
According to the above corollary, it  identifies with
\begin{equation}\label{DesSysIrr}\SYS^{\text{irr}}\stackrel{\sim}{\longrightarrow} T\times \Nu\ ;\ (\curvename_{\boldsymbol t},A)\mapsto ({\boldsymbol t},\boldsymbol\nu=\det(A))\end{equation}
where $\Nu=\C^3_{\boldsymbol\nu}\setminus\{\nu_1^2-4\nu_0\nu_2=0\}$.
Moreover, $\SYS^{\text{irr}}$ can be seen as an open set of the moduli space of systems on curves defined by
 $$\SYS:=\bigcup_{\boldsymbol{t}\in T} \SYS(\curvename_{\boldsymbol{t}}).$$
{
\begin{remark}Note that here we have slightly modified
the definition of $\SYS$  with respect to the introduction: the  parameter ${\boldsymbol t}$ defining the curve  varies not
in the Teichm\"uller space, but in the affine variety $T$ given in \eqref{defTnotTeich}. The Teichm\"uller space $\Teich$ is the universal cover of $T$. The advantage in working with the algebraic
family $(X_{\mathbf{t}})_{\mathbf{t}\in T}$ is that the isomonodromy foliation is defined by algebraic equations, which will allow us to compute transversality.
The monodromy map can still
be defined locally on the parameter space $T$ and the fact that it is a local diffeomorphism obviously does not
depend on the choice of generators for the fundamental group.
\end{remark}
}

\section{Fuchsian systems on $\bbP^1$}\label{Sec:FuchsianSystems}
Following \cite{HL}, we now describe generic $\sli$-connections $(E,\nabla)$ on $\curvename $ as modifications
of the pull-back via the hyperelliptic cover
\begin{equation}\label{HypCover}\hypcover:\curvename \to \bbP^1\ ;\ (x,y)\mapsto x\end{equation}
of certain logarithmic connections $(\underline E,\underline \nabla)$ on $\bbP^1_x$, actually systems.

We denote by ${\CON}(\bbP^1,\boldsymbol{t})$  the moduli space of logarithmic connections
$(\underline E,\underline \nabla)$ on $\bbP^1$ with polar set $\{0,1,t_1,t_2,t_3,\infty\}$ and the following spectral data:
\begin{equation}\label{PolarAndSpectralData}\left\{0\, ,\, -\frac{1}{2}\right\}\ \ \ \text{over}\ \ \ x=0,1,\infty,\ \ \ \text{and}\ \ \
\left\{0\, ,\,\frac{1}{2}\right\}\ \ \ \text{over}\ \ \ x=t_1,t_2,t_3.\end{equation}

We moreover define ${\SYS}(\bbP^1,\boldsymbol{t})$ to be the open subset of ${\CON}(\bbP^1,\boldsymbol{t})$ characterized by following the (generic) properties:
\begin{itemize}
\item $\underline E=\underline E_0$ is the trivial vector bundle;\vspace{.1cm}
\item the $(-\frac{1}{2})$-eigendirections over $x=0,1$ and the $0$-eigendirection over $x=\infty$ are pairwise distinct;\vspace{.1cm}
\item   the $0$-eigendirection over $x=\infty$ is distinct  from the $\frac{1}{2}$-eigendirection over any $x=t_i$ for $i\in \{1,2,3\}$.
\end{itemize}

Now  the eigendirections of any fuchsian system in ${\SYS}(\bbP^1,\boldsymbol{t})$  can be normalized as follows (up to $\SL$-gauge equivalence)
 \begin{equation}\label{ParabolicData}
{\renewcommand{\arraystretch}{2}
\begin{array}{| r | cccccc |}
\hline
\textrm{at the pole } x =  & 0 & 1 & t_1 & t_2 & t_3 & \infty\\ \hline
  \textrm{eigenvalue } \lambda = &-\frac{1}{2}&-\frac{1}{2}&\frac{1}{2}&\frac{1}{2}&\frac{1}{2}&0\\
\textrm{eigendirection for } \lambda &{\renewcommand{\arraystretch}{1}\begin{pmatrix}0\\ 1\end{pmatrix} }&{\renewcommand{\arraystretch}{1}\begin{pmatrix}1\\ 1\end{pmatrix} }&{\renewcommand{\arraystretch}{1}\begin{pmatrix}z_1\\ 1\end{pmatrix}} &{\renewcommand{\arraystretch}{1}\begin{pmatrix}z_2\\ 1\end{pmatrix}} &{\renewcommand{\arraystretch}{1}\begin{pmatrix}z_3\\ 1\end{pmatrix}} & {\renewcommand{\arraystretch}{.5}\begin{array}{c}\textrm{ }\\{\renewcommand{\arraystretch}{1}\begin{pmatrix}1\\ 0\end{pmatrix}}\\ \textrm{ }\end{array}}\\
\hline
\end{array}}\end{equation}
with $(z_1,z_2,z_3)\in \C^3$.
On the other hand, any fuchsian system in ${\SYS}(\bbP^1,\boldsymbol{t})$ with parabolic data \eqref{ParabolicData}   writes (as a connection on the trivial bundle)
\begin{equation}\label{UniversalConnectionKappai}
\underline \nabla=\nabla_0+c_1\Theta_1+c_2\Theta_2+c_3\Theta_3
\end{equation}
with
$$\nabla_0=\diff +\begin{pmatrix}0&0\\ 0&-\frac{1}{2}\end{pmatrix}\frac{\diff x}{x}
+\begin{pmatrix}0& -\frac{1}{2}\\ 0&-\frac{1}{2}\end{pmatrix}\frac{\diff x}{x-1}
+\sum_{i=1}^3\begin{pmatrix} 0& \frac{z_i}{2}\\ 0& \frac{1}{2}\end{pmatrix}\frac{\diff x}{x-t_i}$$
and
$$\Theta_i=\begin{pmatrix}0&0\\ 1-z_i&0\end{pmatrix}\frac{\diff x}{x}
+\begin{pmatrix}z_i& -z_i\\ z_i&-z_i\end{pmatrix}\frac{\diff x}{x-1}
+\begin{pmatrix}- z_i & z_i^2\\ -1& z_i\end{pmatrix}\frac{\diff x}{x-t_i} .$$

For given parameters $(\boldsymbol z,\boldsymbol c)=(z_1,z_2,z_3,c_1,c_2,c_3)\in \C^6$, we denote by
 $\underline\nabla_{\boldsymbol z,\boldsymbol c}$ the corresponding connection (\ref{UniversalConnectionKappai}).
The moduli space ${\SYS}(\bbP^1,\boldsymbol{t})$ defined above is thus parametrized by $\C^6_{\boldsymbol z,\boldsymbol c}$ as follows.
\begin{equation}\label{DefUnderlineSyst}\left\{\begin{matrix}
\C^6_{\boldsymbol z,\boldsymbol c}&\stackrel{\sim}{\to}&{\SYS}(\bbP^1,\boldsymbol{t})&\subset &{\CON}(\bbP^1,\boldsymbol{t})\\
(\boldsymbol z,\boldsymbol c)&\mapsto& (\underline E_0 ,\nabla_{\boldsymbol z,\boldsymbol c})
\end{matrix}\right.\end{equation}

We note that the eigendirections of $\underline\nabla_{\boldsymbol z,\boldsymbol c}$ with respect to the $0$-eigenvalue satisfy
  \begin{equation}\label{ZeroEigenvalues}
{\renewcommand{\arraystretch}{2}
\begin{array}{| r | cc  |}
\hline
\textrm{at the pole } x =  & 0 & 1  \\ \hline
  \textrm{eigenvalue } \lambda = & 0 & 0\\
\textrm{eigendirection for } \lambda &{\renewcommand{\arraystretch}{1.5}\begin{pmatrix}\frac{1}{2\sum_{i=1}^3c_i(1-z_i)}\\ 1\end{pmatrix}  }  & {\renewcommand{\arraystretch}{.5}\begin{array}{c}\textrm{ }\\{\renewcommand{\arraystretch}{1.5}\begin{pmatrix}1+\frac{1}{2\sum_{i=1}^3c_iz_i}\\ 1\end{pmatrix} }\\ \textrm{ }\end{array}}\\
\hline
\end{array}}\end{equation}

\subsection{Hyperelliptic cover}\label{subsec:hyperelliptic}

Given a fuchsian system $(\underline E_0,\underline\nabla)\in{\SYS}(\bbP^1,\boldsymbol{t})$,
we can pull it back to the curve $\curvename =\curvename_{\boldsymbol t}$ via the hyper-elliptic cover $\hypcover:\curvename \to \bbP^1_x$ given in \eqref{HypCover}.
We thus get a logarithmic connection $\widetilde{\nabla}:=\hypcover^*\underline\nabla$ on the trivial bundle  $E_0\to \curvename $,
with poles at the $6$ Weierstrass points with eigenvalues multiplied by $2$:
$$\{0,-1\}\ \ \ \text{over}\ \ \ w_0,w_1,w_\infty,\ \ \ \text{and}\ \ \
\{0,1\}\ \ \ \text{over}\ \ \ w_{t_1},w_{t_2},w_{t_3}$$
($w_i$ denotes the Weierstrass point over $x=i$).
All these poles are ``apparent singular points'' in the sense that
they disappear after a convenient birational bundle transformation.
More precisely, $\widetilde{\nabla}$-horizontal sections have at most a single pole or zero at these points
(depending on the sign of the non-zero-eigenvalue). If $E$ denotes the rank $2$ vector bundle locally generated
by $\widetilde{\nabla}$-horizontal sections, and $\phi:E\dashrightarrow E_0$ the natural birational bundle isomorphism,
then $\nabla:=\phi^*\widetilde{\nabla}$ is a holomorphic connection, by construction. In terms of \cite[section 1.5]{HL} this map $\phi$ is given by negative elementary transformations in the $0$-eigendirections over $x=0,1,\infty$, and positive elementary transformations in the $1$-eigendirections over $x=t_1,t_2,t_3$ of $\widetilde{\nabla}$.
We denote by $\Phi$ the combined map $\Phi =\phi^*\circ \hypcover^* :  (\underline E_0,\underline\nabla) \to (E,\nabla).$
\begin{equation}\label{diagramPhi}\xymatrix{ (E,\nabla)   & \ar@{..>}[l]_{\phi^*}(E_0,\widetilde{\nabla}) \\
 & (\underline E_0,\underline\nabla)  \ar[u]_{\hypcover^*}    \ar[ul]^{\Phi} }\end{equation}
More generally, we can consider a logarithmic connection $(\underline E,\underline\nabla)$
on $\bbP^1$ with  polar set $\{0,1,t_1,t_2,t_3,\infty\}$ and spectral data \eqref{PolarAndSpectralData}; by the same pull-back construction,
we get a holomorphic connection $(E,\nabla)$ on $\curvename $.
Moreover, since the trace connection $\mathrm{tr}(\underline\nabla)=\diff +\frac{1}{2}\diff \log\left(\frac{(x-t_1)(x-t_2)(x-t_3)}{x(x-1)}\right)$
has trivial monodromy after pull-back on $\curvename $, it follows that $\nabla$ is a $\sli$-connection on $E$.
We have thus defined a map between the corresponding moduli spaces of connections
\begin{equation}\label{DefPhi}\Phi:\CON(\bbP^1,\boldsymbol{t})\to\CON(\curvename_{\boldsymbol{t}})\ ;\ (\underline E,\underline\nabla)\mapsto(E,\nabla)\end{equation}
which has been studied by the last two authors in \cite{HL}.

{
\begin{theorem}[\cite{HL}]\label{Th:MainHL} The image of the map $\Phi$ defined in \eqref{DefPhi} is the moduli space $\CON^{\textrm{irr,ab}}(\curvename_{\boldsymbol{t}})$ of $\sli$-connections with irreducible or abelian monodromy. Moreover, the map  $\Phi$ defines a $2$-fold cover of $\CON^{\textrm{irr,ab}}(\curvename_{\boldsymbol{t}})$, unramified over the open set
$\CON^{\text{irr}}(\curvename_{\boldsymbol{t}})\subset\CON^{\textrm{irr,ab}}(\curvename_{\boldsymbol{t}})$ of irreducible connections.
\end{theorem}

}

 Consider the set  $\Sigma_{\boldsymbol{t}}\subset  \SYS(\bbP^1,\boldsymbol{t})$ defined by
$$\Sigma_{\boldsymbol{t}}:=\left\{(\boldsymbol z,\boldsymbol c)\in \C^6 ~|~z_1-z_2=z_2-z_3=c_1+c_2+c_3=0\right\}\subset \C^6_{\boldsymbol z,\boldsymbol c}\simeq \SYS(\bbP^1,\boldsymbol{t}).$$ Note that $\Sigma_{\boldsymbol{t}}\subset  \SYS(\bbP^1,\boldsymbol{t})$ is characterized by the fact that
\begin{itemize}
\item all three $\frac{1}{2}$-eigendirections coincide;\vspace{.1cm}
\item all three $0$-eigendirections over $x=0,1,\infty$ coincide.
\end{itemize}
Indeed, these conditions are equivalent to
$$z_1=z_2=z_3=:z\ \ \ \text{and}\ \ \ c_1+c_2+c_3=0.$$
We shall now prove, in a direct and explicit way, that the set $\Sigma_{\boldsymbol{t}}$ consists in all those fuchsian systems in $ \SYS(\bbP^1,\boldsymbol{t})$  whose image $(E,\nabla)$ under $\Phi$ defines a holomorphic system on $\curvename_{\boldsymbol{t}}$ (i.e.  where $E$ is the  trivial bundle $E=E_0$). Moreover, every irreducible holomorphic system in $\SYS(\curvename_{\boldsymbol{t}})$ can be obtained in that way.

\begin{proposition} With the notation above, we have
\begin{itemize}
\item[$\bullet$] $\Phi( \Sigma_{\boldsymbol{t}} ) \subset \SYS(\curvename_{\boldsymbol{t}})$ and \vspace{.1cm}
\item[$\bullet$]  $\SYS(\curvename_{\boldsymbol{t}})^{\text{irr}} \subset \Phi( \Sigma_{\boldsymbol{t}} ).$
\end{itemize}
\end{proposition}

\begin{proof}
  Consider  the connection $\underline\nabla_{\boldsymbol z,\boldsymbol c}$ defined in  (\ref{UniversalConnectionKappai}) corresponding to a point in $\Sigma_{\boldsymbol{t}}$.
After gauge transformation by the (constant) matrix
$$\begin{pmatrix}1&z\\0&1\end{pmatrix}$$
the connection matrix becomes of the form
\begin{equation}\label{MatrixNabla}
\begin{pmatrix}0&\frac{b(x)\diff x}{x(x-1)}\\ \frac{c(x)\diff x}{\prod_{i=1}^{3}(x-t_i)}&\frac{1}{2}\diff \log\left(\frac{\prod_{i=1}^{3}(x-t_i)}{x(x-1)}\right)\end{pmatrix}\quad \text{with}
\end{equation}
\begin{equation}\label{FormulaBC}
\ \ \ \left\{\begin{matrix}
b(x)&=&\frac{(2z-1)x-z}{2}\hfill\\
c(x)&=&-(t_1c_1+t_2c_2+t_3c_3)x-(t_1t_2c_3+t_2t_3c_1+t_3t_1c_2) .
\end{matrix}\right.\end{equation}
The connection matrix, after lift to $\curvename_{\boldsymbol{t}}$ via $\hypcover$ (see \eqref{HypCover}), can of course also be written as \eqref{MatrixNabla}, we just have to keep in mind that $x$ is not an appropriate local coordinate near a Weierstrass point (but $y$ is, and hence the residues are double, see \eqref{DefCurve}). The birational bundle transformation $\phi$ defined in \eqref{diagramPhi} consists in three negative elementary transformations (in the $0$-eigendirections over $x=0,1,\infty$), and three positive elementary transformations (in the $1$-eigendirections over $x=t_1,t_2,t_3$).
In the particular case \eqref{MatrixNabla} we are considering here, $\phi^*$ is given explicitly by the meromorphic gauge transformation
$$ \begin{pmatrix}1&0\\0&\frac{x(x-1)}{y}\end{pmatrix} , $$
yielding the following matrix connection on the (trivial) bundle $E:=\phi^*E_0\simeq E_0$
\begin{equation}\label{MatrixConnectionFibreTrivial}
A=\begin{pmatrix}0&\beta\\ \gamma&0\end{pmatrix}:=\begin{pmatrix}0&b(x)\frac{\diff x}{y}\\ c(x)\frac{\diff x}{y}&0\end{pmatrix}.
\end{equation}
Following Proposition \ref{PropReduSystem}, the system defined by $A$ is reducible if, and only if,
the holomorphic $1$-forms $\beta$ and $\gamma$ are proportional, i.e. share the same zeros:
\begin{equation}\label{RedLocus}
\frac{z}{1-2z}=\frac{t_1t_2c_3+t_2t_3c_1+t_3t_1c_2}{t_1c_1+t_2c_2+t_3c_3}.
\end{equation}
The explicit restriction map $\Phi\vert_{\Sigma_{\boldsymbol{t}}}:\Sigma_{\boldsymbol{t}}\to\SYS^\times( \curvename_{\boldsymbol{t}})$, where $\SYS^\times( \curvename_{\boldsymbol{t}})\subset \SYS( \curvename_{\boldsymbol{t}})$ denotes the moduli space of all non trivial systems on $\curvename_{\boldsymbol{t}}$ and $\Phi$ is defined by $\Phi=\phi^*\circ \hypcover^*$ as in  \eqref{diagramPhi},  is surjective and \'etale over the open set
$\SYS^{\text{irr}}( \curvename_{\boldsymbol{t}})$ of irreducible systems $(\curvename_{\boldsymbol{t}},A)$. Indeed, let $\boldsymbol{\nu}=\nu_2x^2+\nu_1 x+\nu_0$ be an element of $\SYS( \curvename_{\boldsymbol{t}})$. Let $\boldsymbol{\nu}=\nu \cdot f_\beta \cdot f_\gamma$ be a decomposition, where $\nu \in \mathbb{C}^*$ and $f_\beta,f_\gamma$ are unitary polynomials in $x$ of degree at most $1$. If $f_\beta\neq x-\frac{1}{2}$, then there is a unique choice of $z\in \mathbb{C}$ such that $b(x)$, given by formula \eqref{FormulaBC} is a scalar multiple of $f_\beta(x)$. Furthermore, there is a unique choice of $\mathbf{c}\in \mathbb{C}^3$ such that  $c(x)$, given by formula \eqref{FormulaBC}, equals  $-\frac{\boldsymbol{\nu}(x)}{b(x)}$. Hence if $x=\frac{1}{2}$ is not a root of $\nu$, we get precisely two preimages. Otherwise, we get precisely one preimage, which is however not double (a second preimage would appear in a convenient compactification of $\Sigma_{\boldsymbol{t}}$ in $\CON(\bbP^1,\boldsymbol{t})$). Explicitely, we generically have a decomposition of the form
$\boldsymbol{\nu}=\nu(x-x_\beta)(x-x_\gamma)$, with $2x_\beta-1\not=0$. Then
we get the explicit preimage
\begin{equation}\label{eq:sectionPhi}
c_i=2\nu\frac{t_i-x_\gamma}{(t_i-t_j)(t_i-t_k)}(2x_\beta-1)\ \ \ \ \ \ \text{for}\ \{i,j,k\}=\{1,2,3\}
\end{equation}
$$\text{and}\ \ \ z=\frac{x_\beta}{2x_\beta-1}.$$
\end{proof}

\subsection{Darboux coordinates}\label{Sec:DarbouxCoordinates}
Moduli spaces of connections on curves have a natural (holomorphic) symplectic structure.
In the case of ${\CON}(\bbP^1,\boldsymbol{t})$, the symplectic two-form is given in the affine chart
${\SYS}(\bbP^1,\boldsymbol{t})$ by $\omega=\diff z_1\wedge \diff c_1+\diff z_2\wedge \diff c_2+\diff z_3\wedge \diff c_3$.
The classical Darboux coordinates, that will be needed to describe isomonodromy equations, are defined as follows.

The vector  $e_1={}^t(1,0)$
becomes an eigenvector of the matrix connection $\nabla_{\boldsymbol z,\boldsymbol c}$
for $3$ values of $x=q_1,q_2,q_3$ (counted with multiplicity),
namely at the zeros of the $(2,1)$-coefficient of the matrix connection:
\begin{equation}\label{Equaqk}
\sum_{i=1}^3c_i\frac{(z_i-t_i)x-t_i(z_i-1)}{x-t_i}=\left(\sum_{i=1}^3c_i(z_i-t_i)\right)\frac{\prod_{k=1}^3(x-q_k)}{\prod_{i=1}^3(x-t_i)}
\end{equation}
At each of the three solutions $x=q_k$ of (\ref{Equaqk}), the eigenvector $e_1={}^t(1,0)$ is associated to the eigenvalue
\begin{equation}\label{Equapk}
p_k:=\sum_{i=1}^3c_iz_i\left(\frac{1}{q_k-1}-\frac{1}{q_k-t_i}\right).
\end{equation}
The equations (\ref{Equaqk}) and (\ref{Equapk}) allow us to express our initial variables
$(z_1,z_2,z_3,c_1,c_2,c_3)$ as rational functions of new variables $(q_1,q_2,q_3,p_1,p_2,p_3)$ as follows.
Define
$$\Lambda:=\sum_{\{i,j,k\}=\{1,2,3\}} \frac{p_i(q_i-t_1)(q_i-t_2)(q_i-t_3)}{(q_i-q_j)(q_i-q_k)}.$$
For $i=1,2,3$, denote
$$\Lambda_i:=\Lambda\vert_{t_i=1}$$
the rational function obtained by setting $t_i=1$ in the expression of $\Lambda$.
Then we have, for $\{i,j,k\}=\{1,2,3\}$
\begin{equation}\label{Equacizi}
c_i=-\frac{(q_1-t_i)(q_2-t_i)(q_3-t_i)}{t_i(t_i-1)(t_i-t_j)(t_i-t_k)}\Lambda\ \ \ \text{and}\ \ \
z_i=t_i\frac{\Lambda_i}{\Lambda}.
\end{equation}
The rational map
\begin{equation}\label{Canonical6Cover}
\Darbouxcover:\C^6_{\boldsymbol q,\boldsymbol p}\dashrightarrow \C^6_{\boldsymbol z,\boldsymbol c}\simeq{\SYS}(\bbP^1,\boldsymbol{t})
\end{equation}
has degree $6$: the (birational) Galois group of this map is the permutation group on indices $k=1,2,3$
for pairs $(q_k,p_k)$.
 In these new coordinates, the symplectic form writes
$$\omega=\sum_{i=1}^3\diff z_i\wedge \diff c_i=\sum_{k=1}^3\diff q_k\wedge \diff p_k.$$

\subsection{Isomonodromy equations as a Hamiltonian system}\label{Sec:DarbouxIsom}
We now let the polar parameter $\boldsymbol{t}=(t_1,t_2,t_3)$ vary in the space $T$ defined in \eqref{defTnotTeich}.
A local deformation $\boldsymbol{t}\mapsto(\curvename_{\boldsymbol{t}},E_{\boldsymbol{t}},\nabla_{\boldsymbol{t}})\in\CON(\curvename_{\boldsymbol{t}})$
is said to be isomonodromic if the corresponding monodromy representation $\boldsymbol{t}\mapsto\rho_{\boldsymbol{t}}\in\Rep$
is constant. To define this latter arrow, we can choose and locally follow a system of generators for the fundamental group;
the isomonodromy condition is clearly independent of the choice. In the moduli space of triples
$$\CON:=\bigcup_{\boldsymbol{t}\in T} \CON(\curvename_{\boldsymbol{t}}),$$
isomonodromic deformations parametrize the leaves of a smooth holomorphic foliation $\mathcal F$,
namely the isomonodromic foliation. More precisely, $\Fiso $ has dimension $3$ and is transversal
to the projection $\curvename :\CON\to T$; moreover, holonomy induces symplectic analytic isomorphisms between fibers $\CON(\curvename_{\boldsymbol{t}})$. The foliation $\Fiso $ is also called the non-linear Gauss-Manin connection
in \cite[section 8]{Simpson}, and it is proved there that $\CON$ is quasi-projective and $\Fiso $ defined
by polynomial equations. However, it is difficult to provide explicit equations in $\CON$.

The construction of $\Phi$ in section \ref{subsec:hyperelliptic} can be performed on the monodromy setting
(see \cite[section 2.1]{HL}) and therefore commutes with isomonodromic deformations. In Darboux coordinates,
isomonodromic deformation equations are well-known as the Hamiltonian form of the Garnier system.
Let us recall them explicitely.
For $i=1,2,3$ define $H_i$ by
$$t_i(t_i-1)\prod_{j\not=i}(t_j-t_i)\cdot H_i:=$$
$$\sum_{j=1}^3\frac{\prod_{k\not=j}(q_k-t_i)}{\prod_{k\not=j}(q_k-q_j)}F(q_j)\left(p_j^2-G(q_j)p_j+\frac{p_j}{q_j-t_i}\right),$$
where $F(x)=x(x-1)(x-t_1)(x-t_2)(x-t_3)$ and $G(x)=\frac{F'(x)}{2F(x)}$ (where $F'$ is the derivative with respect to $x$).

Any local analytic map $\boldsymbol{t}\mapsto (\boldsymbol{q}(\boldsymbol{t}),\boldsymbol{p}(\boldsymbol{t}))$ induces, via the $6$-fold cover $\Darbouxcover:\C^6_{\boldsymbol q,\boldsymbol p}\dashrightarrow \C^6_{\boldsymbol z,\boldsymbol c}\simeq{\SYS}(\bbP^1,\boldsymbol{t})$,
a deformation $\boldsymbol{t}\mapsto(\underline E_0,\underline\nabla_{\boldsymbol{z},\boldsymbol{c}})\in{\SYS}(\bbP^1,\boldsymbol{t})$.

\begin{theorem}[Okamoto \cite{Okamoto}]The local deformation induced by $\boldsymbol{t}\mapsto (\boldsymbol{q}(\boldsymbol{t}),\boldsymbol{p}(\boldsymbol{t}))$
is isomonodromic if, and only if,
\begin{equation}\label{GarnierHamiltonian}
\frac{\partial q_k}{\partial t_i}=\frac{\partial H_i}{\partial p_k}\ \ \ \text{and}\ \ \ \frac{\partial p_k}{\partial t_i}=-\frac{\partial H_i}{\partial q_k}\ \ \ \forall i,k=1,2,3.
\end{equation}
\end{theorem}

In other words, the isomonodromic foliation $\underline{\Fiso }$ on ${\SYS}(\bbP^1,*)=\bigcup_{\boldsymbol{t}}{\SYS}(\bbP^1,\boldsymbol{t})$ is defined by the kernel of the $2$-form
$$\Omega=\sum_{k=1}^3\diff q_k\wedge \diff p_k + \sum_{i=1}^3dH_i\wedge dt_i.$$
The tangent space to the foliation is also defined by the $3$ vector fields
\begin{equation}\label{IsomonVectFields}
V_i:=\frac{\partial}{\partial t_i}+\sum_{k=1}^3\left(\frac{\partial H_i}{\partial p_k}\right)\frac{\partial}{\partial q_k}-\sum_{k=1}^3\left(\frac{\partial H_i}{\partial q_k}\right)\frac{\partial}{\partial p_k}
\end{equation}
fir $i=1,2,3$.
Note that the polar locus of these vector fields is given by $$(q_1-q_2)(q_2-q_3)(q_1-q_3)=0,$$ namely the critical
locus of the map (\ref{Canonical6Cover}).

\subsection{Transversality in Darboux coordinates}\label{subsec:CompuTransvDarboux}

Let $\Sigma:=\bigcup_{\boldsymbol{t}}\Sigma_{\boldsymbol{t}}$ denote the locus of those systems
in ${\SYS}(\bbP^1,*)$ that lift under $\Phi$ as a connection $(\curvename ,E,\nabla)$ on the trivial bundle $E=E_0$.
From the characterization described in section \ref{subsec:hyperelliptic}, we have
$$\Sigma:=\{z_1-z_2=z_2-z_3=c_1+c_2+c_3=0\}\subset T\times\C^6_{ \boldsymbol{z},\boldsymbol{c}}.$$
Condition $c_1+c_2+c_3=0$ implies that two of the $q_i$'s are located on $x=0,1$;
after fixing, say $q_1=0$ and $q_2=1$, we can
determine $q_3$ from the matrix connection, as well as all $p_i$'s. In particular, we find $p_3=0$.
Denote
$$\Sigma^{\textrm{Darb}}:=\{q_1=q_2-1=p_3=0\} \subset T\times\C^6_{ \boldsymbol{q},\boldsymbol{p}}.$$
{\bf If we assume} $q_3\not=0,1,\infty$ for the moment (which implies $q_i\not= q_j$ for all $i,j$), the locus of non-trivial connections on the trivial bundle $E_0$ in $\CON $ is parametrized by the image under $\Phi$ of
\begin{equation}\label{eq:paramSigmaDarboux}
\Darbouxcover\vert_{\Sigma^{\textrm{Darb}}}\ :\ {\Sigma^{\textrm{Darb}}}\longrightarrow\Sigma
\end{equation} defined, for $z_1=z_2=z_3=:z$ by
$$(p_1,p_2,q_3)\mapsto\left\{\begin{array}{ccl}
1-\frac{1}{z}&=&\frac{(t_1-1)(t_2-1)(t_3-1)}{t_1t_2t_3}\frac{q_3}{q_3-1}\frac{p_2}{p_1}\vspace{.2cm}\\
c_1&=&\frac{t_1t_2t_3(q_3-1)p_1-(t_1-1)(t_2-1)(t_3-1)q_3p_2}{q_3(q_3-1)}\frac{q_3-t_1}{(t_1-t_2)(t_1-t_3)}\vspace{.2cm}\\
c_2&=&\frac{t_1t_2t_3(q_3-1)p_1-(t_1-1)(t_2-1)(t_3-1)q_3p_2}{q_3(q_3-1)}\frac{q_3-t_2}{(t_2-t_1)(t_2-t_3)}\vspace{.2cm}\\
c_3&=&\frac{t_1t_2t_3(q_3-1)p_1-(t_1-1)(t_2-1)(t_3-1)q_3p_2}{q_3(q_3-1)}\frac{q_3-t_3}{(t_3-t_1)(t_3-t_2)}\, .
\end{array}\right.$$
Note that $c_i=0$  if $q_3=t_i$.  The locus of non-transversality of the isomonodromy foliation with $\Sigma^{\textrm{Darb}}$ is contained in the zero locus of the determinant
\begin{equation}\label{DetFormule}\det\left(V_i\cdot F_j\right)\vert_{\Sigma^{\textrm{Darb}}}=
\frac{t_1t_2t_3(q_3-1)^2p_1+(t_1-1)(t_2-1)(t_3-1)q_3^2p_2}{8(t_1-t_2)(t_2-t_3)(t_1-t_3)q_3^2(q_3-1)^2} ,\end{equation}
where $(F_1,F_2,F_3)=(q_1,q_2,p_3)$.
Reversing the map \eqref{eq:paramSigmaDarboux} above, we get
\begin{equation}\label{eq:reversePi}
\begin{matrix}\Darbouxcover^{-1}\ :\ \Sigma&\longrightarrow&\Sigma^{\textrm{Darb}}\end{matrix} \end{equation} defined by
$$\begin{matrix}(z,c_2,c_3)&\mapsto&(p_1,p_2,q_3) = \left(\frac{z}{t_1t_2t_3}Q_0\, ,\frac{z-1}{(t_1-1)(t_2-1)(t_3-1)}Q_1\, , -\frac{Q_0}{Q_\infty} \right)\, , \end{matrix} $$
where we denote  $$Q_0:=t_2t_3c_1+t_1t_3c_2+t_1t_2c_3\, , \quad Q_{\infty}:=t_1c_1+t_2c_2+t_3c_3, , \quad Q_{1}:=Q_0+Q_{\infty}\, .$$
Note that the set $\{Q_k=0\}$ on $\Sigma$ corresponds to $q_3=k$ in Darboux coordinates.
After substitution, the numerator of the determinant \eqref{DetFormule} vanishes if and only if
\begin{equation}\label{RedLocusNew}
\frac{z}{1-2z}=\frac{Q_0}{Q_\infty}.
\end{equation}
We here recognize the locus given in equation \eqref{RedLocus} where the monodromy is reducible. We have now established the main ingredients for the first proof of Theorem \ref{t:localdiffeo}.

\section{First proof of the main theorem}\label{Sec:Proof1}

Consider the map
$$\begin{matrix}{\SYS}(\bbP^1,*)
&\stackrel{\Phi}{\rightarrow}&\CON  ,
\end{matrix}$$
where we use the notations of the previous sections. In particular, $\CON$ denotes the moduli space of $\sli$-connections over curves of the form $\curvename_{\boldsymbol{t}}$ with $\boldsymbol{t}\in T$, ${\SYS}(\bbP^1,*)$ denotes the moduli space of fuchsian rank $2$ systems over $\bbP^1$ with polar set $\{0,1,t_1,t_2,t_3,\infty\}$ and spectral data \eqref{PolarAndSpectralData}  and $\Phi$ is the hyperelliptic lifting map (see section \ref{subsec:hyperelliptic}). We have seen that the locus $\SYS^{\textrm{irr}}\subset \CON$ of irreducible connections defined on the trivial bundle over curves $\curvename_{\boldsymbol{t}}$ (i.e. irreducible holomorphic systems over genus $2$ curves) is contained in the image   $\Phi ({\SYS}(\bbP^1,*)).$ Moreover,    the preimage under $\Phi$ of the locus $\SYS \subset \CON$ of holomorphic systems over genus $2$ curves is given by
$$\Sigma:=\{z_1=z_2=z_3=:z\, , \, c_1+c_2+c_3=0\}\subset T\times\C^6_{ \boldsymbol{z},\boldsymbol{c}}\simeq {\SYS}(\bbP^1,*) .$$
{
Since $\Phi$ is \'etale over $\SYS^{\textrm{irr}}$, the isomonodromy foliation in $\CON$ is transversal to $\SYS^{\textrm{irr}}$ if and only if the lift of the isomonodromy foliation is transversal to $\Sigma^{\textrm{irr}}$ in ${\SYS}(\bbP^1,*)$, where $\Sigma^{\textrm{irr}}=\Sigma \setminus \Sigma^{\textrm{red}}$ and $\Sigma^{\textrm{red}}$ is the lift of the reducible locus in $\SYS$ calculated in equations \eqref{RedLocus} and \eqref{RedLocusNew}
$$\Sigma^{\textrm{red}}:=\left\{(\boldsymbol{t},\boldsymbol{z},\boldsymbol{c})\in \Sigma ~\middle|~ Q_0 z+Q_1(z-1)+Q_\infty=0\right\} .$$
Now consider the rational map $\Psi|_{\Sigma^{\textrm{Darb}}} : \Sigma^{\textrm{Darb}} \dashrightarrow \Sigma$ in equation \eqref{eq:paramSigmaDarboux}, where $$\Sigma^{\textrm{Darb}}:=\{q_1=q_2-1=p_3=0\} \subset T\times\C^6_{ \boldsymbol{q},\boldsymbol{p}}\, .$$
\begin{lemma}\label{lem41} With respect to the notation above and
$$Q^{\textrm{Darb}}:=q_3(q_3-1)(t_1t_2t_3(q_3-1)p_1-(t_1-1)(t_2-1)(t_3-1)q_3p_2) \, , $$
the restriction of $\Psi|_{\Sigma^{\textrm{Darb}}}$ to $\Sigma^{\textrm{Darb}}\setminus \{Q^{\textrm{Darb}}=0\}$ is a well defined map, taking values in $\Sigma \setminus \{Q_0Q_1Q_\infty=0\}$. This corestricted map is moreover bijective and maps $\Sigma^{\textrm{Darb,red}}\setminus \{Q^{\textrm{Darb}}=0\}$ onto $\Sigma^{\textrm{red}} \setminus \{Q_0Q_1Q_\infty=0\}$, where
  $$\Sigma^{\textrm{Darb,red}}:=\left\{(\boldsymbol{t},\boldsymbol{q},\boldsymbol{p})\in \Sigma^{\textrm{Darb}} ~\middle|~t_1t_2t_3(q_3-1)^2p_1+(t_1-1)(t_2-1)(t_3-1)q_3^2p_2=0\right\}\, . $$
\end{lemma}
\begin{proof}
The polar locus of  $\Psi|_{\Sigma^{\textrm{Darb}}}$ is given by $\{Q^{\textrm{Darb}}=0\}$ according to equation \eqref{eq:paramSigmaDarboux}. On the other hand, we have an inverse of the rational map $\Psi|_{\Sigma^{\textrm{Darb}}}$ given by equation \eqref{eq:reversePi}. The claim now follows from straightforward computation. \end{proof}

\begin{lemma} There is a neighborhood $U$ of $\Sigma^{\textrm{Darb}}\setminus \{Q^{\textrm{Darb}}=0\}$ in $T\times\C^6_{ \boldsymbol{q},\boldsymbol{p}}$  such that $\Psi|_{U} : U\to T\times\C^6_{ \boldsymbol{c},\boldsymbol{z}}$, defined by equation \eqref{Canonical6Cover},  is a local diffeomorphism onto its image. \end{lemma}
\begin{proof}
Since $\Psi|_{\Sigma^{\textrm{Darb}}}$ is well defined outside $\{Q^{\textrm{Darb}}=0\}$, the indeterminacy locus of $\Psi$ does not intersect $\Sigma^{\textrm{Darb}}\setminus \{Q^{\textrm{Darb}}=0\}$. On the other hand, $\Psi$ has a local analytic section (a right inverse) given by equations  (\ref{Equaqk}) and (\ref{Equapk})
and is rational. We have to check that $\Psi^{-1}$ has no indeterminacy points on $\left(\Sigma\setminus \{Q_0Q_1Q_\infty=0\}\right)= \Psi\left( \Sigma^{\textrm{Darb}}\setminus \{Q^{\textrm{Darb}}=0\}\right)$.
Since $c_1z_1+c_2z_2+c_3z_3$ vanishes identically on $\Sigma$, the expression $c_1z_1+c_2z_2+c_3z_3-Q_\infty$ is nowhere vanishing in a neighborhood of $\Sigma\setminus \{Q_0Q_1Q_\infty=0\}$. This implies that equation  (\ref{Equaqk}) provides a unitary polynomial of degree $3$ defining $\prod_{k=1}^3(x-q_k)$ from $(\boldsymbol{t}, \boldsymbol{c}, \boldsymbol{z})$ in this neighborhood. Equation (\ref{Equapk}) defining $(p_1,p_2,p_3)$ from $(q_1,q_2,q_3)$ and $(\boldsymbol{t}, \boldsymbol{c}, \boldsymbol{z})$  however contains indeterminacy summands like
$$\frac{\sum_{i=1}^3c_iz_i}{q_2-1}\ \ \ \text{and}\ \ \ \frac{c_iz_i}{q_3-t_i}.$$
 However, we can modify these expressions by using the expression of $F(x):=(x-q_1)(x-q_2)(x-q_3)$
defined by formula (\ref{Equaqk}) and the fact that $q_2-1=-\frac{F(1)}{(q_1-1)(q_3-1)}$ for instance.
By this way, we actually get an alternate formula for $p_2$ and $p_3$
which is now well defined near $\Sigma\setminus \{Q_0Q_1Q_\infty=0\}$. The claim now follows from Lemma \ref{lem41}. \end{proof}
In the space  $ T\times\C^6_{ \boldsymbol{q},\boldsymbol{p}}$  of Darboux coordinates we have explicit expressions of \begin{itemize}
\item[$\bullet$] the isomonodromy foliation and
\item[$\bullet$] the lift via $\Phi\circ \Darbouxcover$ of the locus $\SYS$ of the trivial bundle in $\CON$: it contains $\Sigma^{\textrm{Darb,irr}}:=\Sigma^{\textrm{Darb}}\setminus \Sigma^{\textrm{Darb,red}}$ as a large open subset (see Section \ref{Sec:DarbouxIsom}).
\end{itemize}
}
We have seen by direct  computation that the isomonodromy foliation is transversal to $\Sigma^{\textrm{Darb,irr}}$. It follows that the isomonodromy foliation is transversal to $\Sigma^{\textrm{irr}}\setminus\{Q_0Q_1Q_\infty=0\}$.
It remains to check transversality for systems corresponding to points in the special subset $\{Q_0Q_1Q_\infty=0\}$ of $\Sigma^{\textrm{irr}}$.
Yet in the definition of ${\SYS}(\bbP^1,{\boldsymbol{t}})$ and ${\CON}(\bbP^1,{\boldsymbol{t}})$ in Section \ref{Sec:FuchsianSystems}, the polar set was split in half according to two types of parabolic data. The fact that we associated the $\left\{0,-\frac{1}{2}\right\}$-type to the poles at $I^-:=\{x=0,1,\infty\}$ and the $\left\{0,\frac{1}{2}\right\}$-type to the poles at $I^+:=\{x=t_1,t_2,t_3\}$  was arbitrary. For any $k\in I^-$, we can perform the same construction setting $I^-:=\{x=0,1,\infty, t_1\}\setminus \{x=k\}$ and $I^+:=\{x=k,t_2,t_3\}$ for instance. The correspondence between the former and the new construction of ${\CON}(\bbP^1,{\boldsymbol{t}})$ is given by a particular birational bundle isomorphism $\phi^{\textrm{mod}}$ on the corresponding logarithmic connections on $\bbP^1$, namely the combination of a negative elementary transformation in the $0$-eigendirection over the pole $x=k$ and a positive elementary transformation in the $(-\frac{1}{2})$-eigendirection over the pole $x=t_1$. Since $\phi^{\textrm{mod}}(\underline{E}_0) \simeq \underline{E}_0$, where $\underline{E}_0$ denotes the trivial bundle on $\bbP^1$ as usual, $\phi^{\textrm{mod}}$ actually defines a reparametrization of ${\SYS}(\bbP^1,{\boldsymbol{t}})$ (and in particular, of $\Sigma^{\textrm{irr}}$), which can also be seen as a Moebius transformation in the base.  One can easily check that a system in the former construction of $\Sigma^{\textrm{irr}}$ corresponding to a point in $\{Q_k=0\}$ is no longer contained in the special subset of the new construction of $\Sigma^{\textrm{irr}}$. In summary, transversality of the isomonodromy foliation to $$\Phi(\Sigma^{\textrm{irr}}\setminus\{Q_0Q_1Q_\infty=0\})\subset \SYS^{\textrm{irr}}$$ implies transversality at any point of $\SYS^{\textrm{irr}}$ by a reparametrization of the family of curves $$\bigcup_{\boldsymbol{t}\in T}X_{\boldsymbol{t}}.$$
This finishes the first proof of Theorem \ref{t:localdiffeo}. Next we proceed to introduce the ingredients for the second proof of Theorem \ref{t:localdiffeo}.

\section{Branched projective structures}\label{s:BPS}
Given a compact oriented topological surface $\curvenametop$, a branched projective structure $\sigma$ over $\curvenametop$ is the data of a covering $\curvenametop=\bigcup U_i$ by open sets and for each $U_i$ a finite branched cover preserving the orientation $\varphi_i:U_i\rightarrow V_i\subset \bbP^1$ such that on any $U_i\cap U_j\neq \emptyset$ there exists a homography $A_{ij}:\bbP^1\rightarrow\bbP^1$ satisfying $$\varphi_{i}=A_{ij}\circ\varphi_j\quad\text{ on }\quad U_i\cap U_j.$$ By abuse of language, we will call each $\varphi_i$ a chart of the branched projective structure.
The branching divisor of $\sigma$ is the divisor $\text{div}(\sigma)=\sum (n_p-1)p$ where $n_p$ is the order of branching of the map $\varphi_i$ around $p\in U_i\subset \curvenametop$.
The complex structure of $\bbP^1$ can be pulled back to $\curvenametop\setminus|\text{div}(\sigma)|$ by imposing that each $\varphi_i$ is holomorphic. The complex structure thus defined extends to a unique complex structure $\curvename $ on $\curvenametop$. Note that any chart of $\sigma$ can be analytically extended along any path in $\curvename $.

 Two  branched projective structures $\sigma_1$ and $\sigma_2$ over $\curvenametop$ are said to be equivalent if there exists a homeomorphism $(\curvenametop,\sigma_1)\rightarrow(\curvenametop,\sigma_2)$ that is projective in the corresponding charts and that induces the identity on the fundamental group.

 If we fix a universal covering map $\widetilde{\curvenametop}\rightarrow \curvenametop$ of $\curvenametop$ with $\fundgroup(\curvenametop)=\text{Aut}(\widetilde{\curvenametop}|\curvenametop)$, we can associate a class of equivariant maps defined on $\widetilde{S}$ to any branched projective structure on $S$, called {\em developing map}. Indeed, take a chart $\varphi$ of $\sigma$, we define a  devoloping map $\mathcal{D}:\widetilde{\curvenametop}\rightarrow\bbP^1$ by extending the chart $\varphi$ along paths in $\curvenametop$. The map $\mathcal{D}$ is equivariant with respect to a representation $\rho:\fundgroup (\curvenametop)\rightarrow\text{PSL}_2(\mathbb{C})$, i.e. for each $\gamma\in\fundgroup (\curvenametop)$, $$\mathcal{D}(\gamma\cdot z)=\rho(\gamma)\circ\mathcal{D}(z).$$
The map $\mathcal{D}$ is holomorphic with respect to the lift $\widetilde{\curvename}$ of the complex structure $\curvename$ on $\curvenametop$ to $\widetilde{\curvenametop}$.
Reciprocally, given a complex structure $\curvename $ on $\curvenametop$ we can lift it to a complex structure  $\widetilde{\curvename}$ on $\widetilde{\curvenametop}$. A holomorphic map $\mathcal{D}:\widetilde{\curvename }\rightarrow\bbP^1$ that is equivariant with respect to some representation $\rho:\fundgroup (\curvenametop)\rightarrow\text{PSL}_2(\mathbb{C})$ induces  a branched projective structure on $\curvenametop$ locally defined by $\mathcal{D}$.

From now on, we fix a universal covering map $\widetilde{\curvenametop}\rightarrow \curvenametop$. At the level of developing maps, equivalence of branched projective structures can be read in the following terms: if $A\in\text{PSL}_2(\mathbb{C})$ and $\mathcal{D}$ is a developing map of a branched projective structure with equivariance $\rho$, then \begin{equation}
\mathcal{D}'=A\circ\mathcal{D}\label{eq:developing maps} \end{equation} defines an equivalent branched projective structure with equivariance \begin{equation}\rho'=A\circ\rho\circ A^{-1}.\label{eq:equivariance}\end{equation} Reciprocally, any pair of developing maps $(\mathcal{D},\rho)$, $(\mathcal{D}',\rho')$ associated to equivalent branched projective structures satisfies equations (\ref{eq:developing maps}) and (\ref{eq:equivariance}) for some $A\in\text{PSL}_2(\mathbb{C})$. In particular, two developing maps $\mathcal{D}$ and $\mathcal{D}'$ obtained from different charts of the same $\sigma$ satisfy this last equivalence. Although developing maps are not in general unique, we have that if $\sigma$ is a branched projective structure for which the image of the equivariance is a group with trivial centralizer, then for each representation $\rho:\fundgroup (\curvenametop)\rightarrow\text{PSL}_2(\mathbb{C})$ in the conjugacy class, there is a unique developing map $\mathcal{D}$ associated to $\sigma$ having equivariance $\rho$.

{
Another natural way of constructing branched projective structures is to consider compact curves in complex manifolds $V$ that are generically transverse to a codimension one (singular) holomorphic foliation that is transversely projective. Here, by transversely projective foliation of $V$, we mean  a system of local submersions on $V$, with values in $\mathbb P^1$, which are first integrals of the foliation and which are well-defined up to post-composition by an automorphism of $\mathbb P^1$.  Indeed, the restriction of the local submersions to a smooth compact holomorphic curve $\curvename \subset V$ that is generically transverse to the foliation and avoids its singular set, defines a branched projective structure. Its branching divisor coincides with that of tangencies between the curve and the foliation. The first examples of such a construction are (regular) Riccati foliations, that is, regular holomorphic foliations transverse to $\bbP^1$-bundles over compact curves. They are actually transversely projective, if we consider as local submersions the maps defined on local trivializations of the bundle as ``projection \emph{along the foliation} to a $\bbP^1$-fibre''. Each non-invariant holomorphic section of the bundle induces a branched projective structure on the base curve of the bundle.
}
 In fact, any projective structure on a complex structure $\curvename\in \Teich $ occurs in this way: given a holomorphic developing map $\mathcal{D}:\widetilde{\curvename }\rightarrow\bbP^1$ of a branched projective structure with equivariance $\rho$ we can consider the section of the $\bbP^1$-bundle $\curvename \times_{\rho}\bbP^1\rightarrow \curvename $ defined by $\mathcal{D}$. It defines a curve that is generically transverse to the Riccati foliation induced on  $\curvename \times_{\rho}\bbP^1$ by the horizontal foliation on $\widetilde{\curvename }\times\bbP^1$. Remark that, if $\rho(\pi_1(S))$ has trivial centralizer, the uniqueness of the developing map associated to $\rho$ implies that two different sections of $\curvename \times_{\rho}\bbP^1$ define different branched projective structures on the same complex structure $X$.

Given a natural number $k\in\mathbb{N}$ we define $\mathcal{M}_k$ to be the set of equivalence classes of branched projective structures $\sigma$ on $\curvenametop$ having $k$ critical points counted with multiplicity (i.e. $\text{deg(div}(\sigma))=k$) whose equivariance has trivial centralizer. We have a natural ``monodromy'' map \begin{equation} \label{eq:monodromy of bps} \mathcal{M}_k\rightarrow \text{Hom}(\pi_1(S),\text{PSL}_2(\mathbb{C}))//\text{PSL}_2(\mathbb{C})\end{equation} that associates to each $\sigma$ the point in the $\text{PSL}_2(\mathbb{C})$-character variety associated to its equviariance homomorphism. The fiber over a given representation $\rho:\fundgroup (\curvenametop)\rightarrow\text{PSL}_2(\mathbb{C})$ is denoted by $\mathcal{M}_{k,\rho}$.

\begin{theorem} [\cite{CDF}] \label{t:BPS}
  If $S$ has genus $g\geq 2$, $\rho$ has trivial centralizer and $\mathcal{M}_{k,\rho}\neq\emptyset$\footnote{If the representation is non elementary, then $\mathcal M_{k,\rho}$ is non empty when $k$ is even and $\rho$ lifts to $\text{SL}(2,\mathbb C)$ or when $k$ is odd and $\rho$ does not lift to $\text{SL}(2,\mathbb C)$, see \cite{GKM}. The precise condition when $\rho$ is not elementary but has trivial centralizer is not known, as far as we know.}, then it admits the structure of a complex manifold of dimension $k$.
\end{theorem}

The complex structure on $\mathcal M_{k,\rho}$ comes from the deformation theory of projective structures. Namely, assume that $Y$ is a complex analytic space and that we have the following data :
\begin{enumerate}
\item[a)]  a holomorphic submersion $\Pi : Z\rightarrow Y$ having as fibers simply connected Riemann surfaces,
\item[b)] a free proper discontinuous action of $\Gamma = \pi_1(S)$ on $Z$ preserving the fibration $\Pi$,
\item[c)] a holomorphic map $\mathcal D : Z\rightarrow \mathbb P^1$ which is $\rho$-equivariant with respect to the $\Gamma$-action, and which is a local diffeomorphism on each fiber apart from $k$ orbits counted with multiplicity.
\end{enumerate}
Then, identifying the fibers with the universal covering of $S$ in an equivariant way, we get a map from $Y$ to $\mathcal M_{k,\rho}$ which is holomorphic. In fact, $\mathcal M_{k,\rho}$ is the universal complex space having properties (a)-(c). The fact that this latter is a smooth complex manifold comes from the fact that $\mathcal M_{k,\rho} $ is modelled on Hurwitz spaces, as is shown in the appendix of \cite{CDF}. A particular system of charts showing this will be detailed  in the proof of Lemma \ref{l:tangent bundle}, Section \ref{s:the tangent bundle}.

Extending the ideas of the proof we can glue all the manifolds in Theorem \ref{t:BPS} sharing the genus and the integer $k\geq 0$ to form a complex foliated manifold:
\begin{theorem}\label{t:isomonodromy foliation}
If $S$ has genus $g\geq 2$ and $k\geq 0$, then the space $\mathcal{M}_k$  admits a complex structure compatible with those of Theorem \ref{t:BPS} and for which the monodromy map (\ref{eq:monodromy of bps}) is a holomorphic submersion.
\end{theorem}
The regular holomorphic foliation induced by this monodromy map will be referred to as the isomonodromy foliation on $\mathcal{M}_k$. Since the proof  of Theorem \ref{t:isomonodromy foliation} does not shed much light on that of Theorem \ref{t:localdiffeo}  we leave it for the  Appendix of this paper.

\section{Dictionary between Systems and Rational curves in $\mathcal{M}_{2g-2,\rho}$}
\label{s:rational curves correspond to systems}

Let $S$ be an orientable compact surface of genus $g\geq 2$ and $X\in\Teich$. Take a  system
$(X,A)\in\SYS^{\text{irr}}$  having monodromy $\rho_A:\fundgroup (\curvenametop)\rightarrow\text{SL}_2(\mathbb{C})$. It induces on the trivial bundle the linear connection $\nabla = \mathrm{d} + A$, whose flat sections satisfy  \eqref{eq:linear system}. Its projectivization gives a flat connection $\nabla'$ on the trivial $\mathbb P^1$-bundle $X\times \mathbb P^1$.
{
The foliation $\mathcal F$ induced by $\nabla'$ is Riccati, hence transversely projective.
}
For every $p\in \mathbb P^1$, it induces a branched projective structure $\sigma_A(p)$ on the horizontal $X\times p$ (notice that no horizontal section is flat, since otherwise the system would be reducible). The foliation $\mathcal F$ has been considered by Drach in another context
(see \cite{Drach}). If $B : \widetilde{X} \rightarrow \text{SL} (2,\mathbb C)$ denotes the fundamental matrix of \eqref{eq:linear system} whose value at a point $x_0$ is the identity -- it satisfies that solutions to \eqref{eq:linear system} takes values $Z$ at $x_0$ and $B(x) Z$ at $x$ -- a developing map for the structure $\sigma_A(p)$ is defined as
\begin{equation} \label{eq:variation developing maps} \mathcal D _ p (x )  = B(x)^{-1} (p) . \end{equation}
This latter is $\rho$-equivariant, where $\rho$ is the composition of $\rho_A$ with the natural projection $\text{SL}(2,\mathbb C) \rightarrow \text{PSL} (2,\mathbb C)$, and where we use the natural identification of $X\times p$ with $X$. In particular, the equivariance of $\sigma_A(p)$ is the conjugacy class of $\rho$. Moreover, the critical points of  $\sigma_A(p)$ are the tangency points between $X_p=X\times p$ and $\mathcal{F}$, with the same multiplicities; by the formula~\cite[Proposition 2, p. 26]{Br}
\begin{equation} \label{eq:number of tangencies} | \text{Tang}(\mathcal{F},X_p)| =N_{\mathcal F} \cdot X_p-\chi(X_p)= X_p^2 - \chi (X_p) = 2g-2, \end{equation}
where points in $\text{Tang}(\mathcal F, X_p)$ are counted with multiplicities. Thus, we have built a well-defined map
\begin{equation}\label{DefSigmaA} \sigma_A:\bbP^1\rightarrow \mathcal{M}_{2g-2,\rho} ; \quad \sigma_A(p):=[\sigma(A,p)]. \end{equation}
Now, equation  \eqref{eq:variation developing maps} implies that $\sigma_A$ is holomorphic  with respect to the complex structure given by Theorem \ref{t:BPS} (remark that the irreducibility of $\rho_A$ implies that $\rho$ has trivial centralizer).

\begin{lemma}\label{l:rational curve}
The map  $\sigma_A:\bbP^1\rightarrow \mathcal{M}_{2g-2,\rho}$ defined in \eqref{DefSigmaA}  is an embedding.
\end{lemma}

\begin{proof}
First remark that the irreducibility of $\rho_A$ implies that for every $\sigma\in \mathcal M_{2g-2,\rho}$, there is a unique $\rho$-equivariant developing map $\mathcal D_\sigma$. In the case of $\sigma_A(p)$, this developing map is given by the formula \eqref{eq:variation developing maps}, and we have
$\mathcal D_{\sigma_A(p)} (x_0) = p$. This shows that $\sigma_A$ is injective.

To see that it is an immersion, recall that we can glue all the developing maps $\mathcal D_\sigma $ for $\sigma \in \mathcal M_{2g-2,\rho}$ in the following data: a fibration $\pi : Z \rightarrow \mathcal M_{2g-2,\rho}$ with simply connected Riemann surfaces as fibers, a properly discontinuous action of $\pi_1(S) $ on $Z$ preserving the fibration, and a $\rho$-equivariant map $\mathcal D : Z \rightarrow \mathbb P^1$  which induces on each fiber $\pi^{-1} (\sigma)$ the developing map $\mathcal D_\sigma$. Notice that the fibration $\pi$ is the pull-back of the universal cover of the universal curve over Teichm\"uller space by the natural map $\mathcal M_{2g-2,\rho} \rightarrow \Teich$. In particular, for any $p_0\in \mathbb P^1$, we can choose a germ of section $\sigma \in \mathcal M_{2g-2,\rho}\rightarrow x(\sigma) \in Z$ of $\pi$ defined at the neighborhood of $\sigma_A(p_0)$, taking every $\sigma_A(p)$ (for $p$ close to  $p_0$) to the point $x_0\in \pi^{-1} (\sigma_A(p)) \simeq \widetilde{X}$. For $\sigma$ in a neighborhood of $\sigma_A(p_0)$, we have a well-defined holomorphic function $f (\sigma ) = \mathcal D ( x (\sigma) ) $. By construction we have
$$f \circ \sigma_A (p) = \mathcal D(x(\sigma_A(p))) = \mathcal D _{ \sigma_A(p)} (x_0) = p,$$
for $p$ close to $p_0$, so $\sigma_A$ is an immersion.
  \end{proof}

The following result is not necessary for the proof of Theorem~\ref{t:localdiffeo}. However it completes the proof of the dictionary between systems and rational curves in the isomonodromic moduli spaces $\mathcal M_{2g-2,\rho}$. An infinitesimal version of it will be used in Section \ref{r:counter-example} to provide counter-examples to Theorem \ref{t:localdiffeo} in genus $\geq 3$.
{
\begin{lemma}\label{l:fibers over Teich}
  If a fiber of the  natural (holomorphic) forgetful map $\mathcal{M}_{2g-2,\rho}\rightarrow\Teich$ contains more than two points, it is precisely one of the rational curves of Lemma \ref{l:rational curve}.
\end{lemma}
\begin{proof}
Fix  $X\in\Teich$. Any point in $\mathcal{M}_{2g-2,\rho}$ that projects to $X$ can be thought of as a section of $X\times_{\rho}\mathbb{P}^1\rightarrow X$ of zero self-intersection, thanks to formula \eqref{eq:number of tangencies}.

Suppose there are more than two points in $\mathcal{M}_{2g-2,\rho}$ over $X$. We deduce that there exist three  sections of $X\times_{\rho}\mathbb{P}^1$ of zero self-intersection. Since $\rho$ can be lifted to $\text{SL}_2(\mathbb{C})$ the bundle $X \times _{\rho} \mathbb P^1 \rightarrow X$ is diffeomorphic to a product (see for instance \cite{Goldman}). Under this identification, each section is homologous to a horizontal curve, which implies that the three chosen sections do not intersect pairwise. With those at hand we can trivialize the $\mathbb{P}^1$-bundle biholomorphically  to $X\times\mathbb{P}^1$ taking each section to a horizontal. Now, every horizontal section $X\times p$ with $p\in\mathbb{P}^1$ defines an element in $\mathcal{M}_{2g-2,\rho}$, and there are no other sections of zero self-intersection in $X\times\mathbb{P}^1$. Therefore, the fibre of $\mathcal{M}_{2g-2,\rho}\rightarrow\Teich$ over $X$ is precisely one of the rational curves obtained in Lemma \ref{l:rational curve}.
\end{proof}

As a consequence of Lemma \ref{l:fibers over Teich} we get the following description of parametrized curves in $\mathcal{M}_{2g-2,\rho}$
\begin{corollary} \label{l:rational curves} Let $\rho:\fundgroup (\curvenametop)\rightarrow\mathrm{PSL}_2(\mathbb{C})$ be an irreducible representation that lifts to $\mathrm{SL}(2,\mathbb C)$.\footnote{Otherwise the space $\mathcal M_{2g-2,\rho}$ is empty, see \cite{CDF}.}
  Given a compact connected curve $C$ and a non constant holomorphic map $R:C\rightarrow\mathcal{M}_{2g-2,\rho}$, there exists $(X,A)\in\systems$ and a meromorphic map $R_0:C\rightarrow\bbP^1$ such that $$R=\sigma_A\circ R_0.$$ In particular the image of $R$ is an embedded rational curve.
\end{corollary}

\begin{proof}
First remark that the composition of $R$ with the  natural (holomorphic) forgetful map $\mathcal{M}_{2,\rho}\rightarrow\Teich$ must be some constant $X\in\Teich$, since $\Teich$ has the structure of a bounded domain of holomorphy. Therefore $R(C)$ is contained in a fiber of $\mathcal{M}_{2g-2,\rho}\rightarrow\Teich$. Since $R(C)$ is connected and not a point it must be one of the rational curves of Lemma \ref{l:fibers over Teich}. The result follows.
\end{proof}

}

From Lemmas \ref{l:rational curve} and \ref{l:rational curves}  we deduce that an isomonodromic deformation of systems of monodromy $\rho$ corresponds to a non-trivial deformation of rational curves in $\mathcal{M}_{2g-2,\rho}$.

\section{The tangent bundle to $\mathcal{M}_{k,\rho}$}
\label{s:the tangent bundle}

In this section we fix  $k>0$ and a representation $\rho:\fundgroup (\curvenametop)\rightarrow \text{PSL}_2(\mathbb{C})$ with trivial centralizer such that  $\mathcal{M}_{k,\rho}$ is  non-empty. Whenever there is no risk of confusion, we note $$\mathcal{M}=\mathcal{M}_{k,\rho}.$$ We will show that $T\mathcal{M}$ can be thought of as a push-forward of a line bundle.

Let $\Pi: \mathcal{C}\rightarrow\mathcal{M}$ denote the universal curve bundle over $\mathcal{M}$, that is, the fibre of $\Pi$ over $\sigma\in\mathcal{M}$ is the Riemann surface $X$ underlying $\sigma$. By the results of \cite{CDF}, there exists a transversely projective holomorphic foliation $\mathcal{G}$ of codimension one on $\mathcal{C}$ that induces on each fibre $\Pi^{-1}(\sigma)$ the projective structure $\sigma$.\footnote{With the notation used just after Theorem \ref{t:BPS}, $\mathcal G$ is the quotient of the foliation whose leaves are the level sets of $\mathcal D$ on the fibration $Z\rightarrow \mathcal M_{k,\rho}$ by the action of $\pi_1(S)$.} The points of tangency between $\mathcal{G}$ and a fiber of $\Pi^{-1}(\sigma)$ correspond to the points where $\sigma$ has critical points. Let $B=\text{Tang}(\mathcal{G},\Pi)$ denote the tangency divisor between $\mathcal{G}$ and the fibration $\Pi$. By construction, $\Pi_{|B}:B\rightarrow \mathcal{M}$ is a branched cover of degree $k$  and $\mathcal{G}$ is generically transverse to $B$.

Around a regular value $\sigma\in\mathcal{M}$ of $\Pi_{|B}$ we can define $k$ regular transverse one dimensional holomorphic foliations, namely the push forward of the restriction of $\mathcal{G}$ to $B$ around each of the $k$  points in the fiber $\Pi^{-1}_{|B}(\sigma)$. With the help of these foliations we will analyze the sheaf of sections of $T\mathcal{M}$. Let us formalize this argument.

We know that a line bundle over the source space of a branched covering of degree $k$ can be pushed forward to a rank $k$-vector bundle over the target space of the branched cover. Let  $N_{\mathcal{G}|B}$ denote the restriction of the normal bundle of  $\mathcal{F}$ to $B$. Its push forward $$E:=(\Pi_{|B})_*(N_{\mathcal{G}|B})$$ by $\Pi_{|B}$ is a rank $k$ vector bundle over $\mathcal{M}$.

\begin{lemma}\label{l:tangent bundle} With the above notations, we have $$E\simeq T\mathcal{M}.$$ \label{l:pushforward}\end{lemma}

\begin{proof}
  From the appendix in \cite{CDF} we know how to describe local charts around a point in $\sigma\in\mathcal{M}_{k,\rho}$ for any $\rho$ with trivial centralizer and $k$ such that $\mathcal{M}_{k,\rho}\neq \emptyset$. Let us recall the construction. For each point $q$ in the support of $\text{div}(\sigma)$, consider a chart of the projective structure around it defined on a disc
  $D_q\subset \curvenametop$, and fix a point $Q\in\partial D_q$ and its image in $\mathbb{P}^1$. Let $n_q+1=n+1$ be the order of the projective structure around $q$ and consider a conformal equivalence where the chart becomes the map $z\mapsto z^{n+1}$ defined on $\mathbb{D}$, $q=0$ and $Q=1$. Consider $a=(a_1,\ldots,a_n)\in(\mathbb{C}^{n},0)$  and the polynomial in the variable $z$ \begin{equation}\label{eq:coordinates} P_a(z)=z^{n+1}+a_nz^{n}+\ldots+a_1z-(a_n+\ldots+a_1).\end{equation}  Then $P_a(1)=1$ and its critical values lie on $\mathbb{D}$. $P_0$ corresponds to the normalized covering associated to $\sigma$ at $q$. Let $\mathbb{D}_a:=P_a^{-1}(\mathbb{D})$.  We can topologically glue the covering $P_{a|\mathbb{D}_a}:\mathbb{D}_a\rightarrow\mathbb{D}$ to the projective structure induced by $\sigma$ on the complement of $D_q$ by gluing the boundaries with the help of the equivalence between the values of $P_0$ and $P_a$. By doing such a surgery independently at each critical point of $\sigma$ we obtain a projective structure for each choice of $n_q$ parameters at each $q$. Thus, the family of all such projective structures depends on $k=\sum n_q$ parameters. With the choices involved, it is a chart in $\mathcal{M}_{k,\rho}$. In this chart, the tangent space to $\mathcal{M}_{k,\rho}$ is just the tangent space to the cartesian product of the parameters $(a_1,\ldots,a_{n_q})$ at each critical point $q$. With these coordinates we can define local coordinates of $\mathcal{C}$ around a point $q\in \Pi_{B}^{-1}(\sigma)$, where $\sigma$ has a critical point of order $n=n_q$. Indeed, a neighbourhood corresponds to a neighbourhood of $0$ in the coordinates $(a_1,a_2,\ldots,a_n,z)=(a,z)\in\mathbb{C}^{n+1}$, the projection $\Pi$ restricted to the neighbourhood is just the map $(a,z)\mapsto a$, the set $B$ corresponds to the smooth analytic set $$B\simeq \left\{(a,z)\in\mathbb{C}^{n+1}~\middle|~\frac{\partial P_a}{\partial z}(z)=(n+1)z^{n}+na_nz^{n-1}+\ldots+a_1=0\right\}$$ and the foliation $\mathcal{G}$ has a first integral $v=P(a,z):=P_a(z)$ in the neighbourhood of $0$.
  Let $\mathcal{G}_{|B}$ be the foliation restricted to $B$. Recall that local sections of the normal bundle to a codimension one foliation can be interpreted as tangent vectors to the values of a first integral of the foliation.
  The equivalence in the statement of the lemma will follow from analyzing  the relation that $P_{|B}$ induces between the parameters of $\mathcal{M}_{k,\rho}$ and those of its values. Let $(a,c(a))\in B$ be one of the critical points of $P_a$ and $v=P_a(c(a))$ denote its critical value. Differentiation gives $$\mathrm{d}v=\sum_{i=1}^n(c(a)^i-1)\mathrm{d}a_i$$ and hence for $(a,z)\in B$ and $i=1,\ldots, n$ \begin{equation} \label{eq: isom}\frac{\partial}{\partial{a_i}}= (z^i-1)\frac{\partial}{\partial v}.\end{equation}

Recall that if $\sigma\in\mathcal{M}_{k,\rho}$ is a regular value of  $\Pi_{|B}$, the local holomorphic sections of the bundle  $E$ around a point $\sigma$  are given by direct sums of local sections of the normal bundle to the foliation at each of the preimages of $\sigma$ by $\Pi_{|B}$. By the definition of the charts $(a_1,\ldots a_k)$ of $\mathcal{M}_{k,\rho}$, at such a $\sigma$ each coordinate function  corresponds to a different point of the preimage and a coordinate of type $P_{a_i}(z)=z^2+a_iz-a_i$ around it. Thus, each $\frac{\partial}{\partial a_i}$ corresponds to the section of the normal bundle to the foliation that is zero around all preimages, except for the chart corresponding to the $a_i$ variable where the equivalence (\ref{eq: isom}) takes place (with $n=1$ since the branch point is simple).

To see that this equivalence defines an isomorphism of the sheaves of sections of $T\mathcal{M}$ and of $E$ we first remark that any germ of holomorphic function $f$ on $B$ at the point $q=(0,0)\in B$ of order $n$ can be uniquely written as a sum \begin{equation} \label{eq:weierstrass}\sum_{i=1}^n f_i(a_1,\ldots, a_n)(z^i-1)\quad \text{ for each }(a,z)\in B\end{equation} where each $f_i$ is holomorphic in $(a_1,\ldots,a_n)\in(\mathbb{C}^n,0)$. This follows from the Weierstrass Division Theorem applied to any holomorphic extension $F$ of $f$ to a neighbourhood of $q$ in $\mathcal{C}$ and the Weierstrass polynomial $W:=\frac{1}{n+1}\frac{\partial P_a}{\partial z}(z)$. Indeed, there exist holomorphic functions $\tilde{f}_i(a_1,\ldots,a_n)$ such that $$F(a,z)=\sum_{i=0}^{n} \tilde{f}_i(a_1,\ldots,a_n) z^i + W(a,z) Q(a,z)$$ for some holomorphic $Q\in\mathcal{O}_{(\mathbb{C}^{n+1},0)}$. Defining $f_i=\tilde{f}_i$ for $i>0$ and $f_0=\tilde{f}_0+\sum_{i=1}^{n} \tilde{f}_i$ and restricting to $B=\{W=0\}$ we get the desired expression (\ref{eq:weierstrass}) for the germ  $f:(B,q)\rightarrow\mathbb{C}$.

With this at hand we deduce that, outside the branching divisor of $\Pi_{|B}$, the correspondence $\frac{\partial}{\partial a_i}\mapsto (z^i-1)\frac{\partial}{\partial v_i}$ defines the desired isomorphism.

Here we just treat the case where the critical fibers contain a single point, but the general case follows from the same type of argument. In the critical fibers we have a unique  point $q\in B$ of order $n=k$ and we can consider a chart given by a polynomial as in (\ref{eq:coordinates}). We need to extend the isomorphism to a neighbourhood of this point. The correspondence is already given by (\ref{eq:   isom}). To a holomorphic vector field $\sum f_i(a_1,\ldots,a_n)\frac{\partial}{\partial a_i}$ around $\sigma=(0,\ldots, 0)$ we associate $$\sum_{i=1}^n f_i(a_1,\ldots,a_n)\left((z^i-1)\frac{\partial}{\partial v}\right)$$ around the unique point $q=(0,0)$ in the fiber of $\Pi_{|B}$ over $\sigma$. This last expression corresponds uniquely to the vector field $f(a,z)\frac{\partial}{\partial v}$ where $f$ is defined by (\ref{eq:weierstrass}) in a neighbourhood of $q$ in $B$. Since every holomorphic function on $B$ can be uniquely written in this way, the equivalence is a sheaf isomorphism.
\end{proof}

\section{Rigidity of rational curves in $\mathcal{M}_{2,\rho}$} \label{s:rigidity}
Suppose $\mathcal{R}$ is a smooth embedded rational curve in $\mathcal{M}=\mathcal{M}_{k,\rho}$. Then the bundle $T\mathcal{M}_{|\mathcal{R}}$ splits as a direct sum of $k$ line bundles  over $\mathcal{R}$: $$T\mathcal{M}_{|\mathcal{R}} = \bigoplus_{i=1}^k \mathcal{O}(d_i) .$$ Among the factors, there is one coming from  $\Pi_*(\Pi^*(T\mathcal{R}))$ that implies without loss of generality that $d_1$ can be assumed to be $2$. The rigidity of the curve $\mathcal{R}$ in $\mathcal{M}_{k,\rho}$ is equivalent to proving that  $d_i<0$ for every $i\geq 2$. 

The Chern class of the determinant bundle $\wedge^2 T\mathcal{M}_{|\mathcal{R}}$ is the sum $2+(d_2+\cdots+d_k)$. In particular, when  $k=2$, rigidity is equivalent to showing that this sum  is at most $1$. In fact the $d_2$ corresponds to the self-intersection of the curve $\mathcal{R}$ in this case.
The problem is settled in the following statement.

\begin{proposition}\label{p:negative selfintersection} Every smooth rational curve in $\mathcal{M}_{2,\rho}$ has  self-inter\-section $-4$.
\end{proposition}
\begin{proof}
   By  Lemma \ref{l:pushforward}, we have $c_1(\wedge^2 T\mathcal{M}_{|\mathcal{R}})=c_1(\wedge^2 E_{|\mathcal{R}})$. The latter can be calculated using the following relation between the Chern class of a line bundle and that of the determinant of the push forward by a branched covering.

\begin{lemma}\label{l:groethendieck_riemann_roch} If $L\rightarrow C$ is a line bundle over a smooth curve $C$ and $f:C\rightarrow C'$ is a branched cover of degree $2$, then $$c_1(\wedge^2 f_* L)=c_1(L)-\frac{1}{2}\sum (e_p-1)$$ where $e_p$ denotes the degree of $f$ around $p\in C$.\end{lemma}

Note that  the number $\sum(e_p-1)$ is always even.

In our case, we take $C=\Pi^{-1}(\mathcal{R})\cap B$, $L=N_{\mathcal{G}|C}$, $C'=\mathcal{R}$ and $f=\Pi_{|C}$. The curve $C$ is the projectivization of the set of eigenvectors of $A$ in the product $X\times \mathbb P^1$, so we will call it the eigencurve in the sequel. By Corollary \ref{l:rational curves},  the rational curve $\mathcal{R}$ can be parametrized by some embedding $\sigma_A:\bbP^1\rightarrow \mathcal{M}$ which comes from a regular holomorphic foliation $\mathcal{F}$ on some $X\times\bbP^1$ that is transverse to the $\bbP^1$-fibers. In fact, by construction, the restriction of $\Pi$ and $\mathcal{G}$ to  $\Pi^{-1}(\mathcal{R})$ are naturally equivalent to the fibration $\Pi_2:X\times\bbP^1\rightarrow \bbP^1$ endowed with the foliation $\mathcal{F}$. Under this equivalence, $C=\text{Tang}(\mathcal{F},\Pi_2)$ and $$L=N_{\mathcal{F}|C}\simeq N(\Pi_2)_{|C}\simeq T(\Pi_{1})_{|C}\simeq \Pi_{1|C}^{*}(T(\bbP^1))$$ where $\Pi_1 :X\times\bbP^1\rightarrow X$ is the canonical projection. Therefore, since $\Pi_{1|C}$ has degree two, we have $c_1(L)=4$.

On the other hand, we need to calculate the degree of the branching divisor of $f=\Pi_{2|C}:C\rightarrow\bbP^1$.  Since the branched covering has degree two, the only possibility is that the branching points are simple. Thus it suffices to count the number of fibers of type $X_p=X\times p$ such that $\mathcal{F}$ has a unique point of tangency with  $X_p$ of multiplicity two.
\begin{lemma}\label{l:nongeneric horizontal fibers}
  There are twelve points $p_1,\ldots, p_{12}\in\bbP^1$ counted with multiplicity such that $X_{p_i}$ and $\mathcal{F}$ have a unique point of tangency of multiplicity two.
\end{lemma}
\begin{proof}
  The Riccati equation $\mathcal{F}$ on $X\times\bbP^1$ can be written as $dy=\alpha y^2+\beta y+\gamma$ where $\alpha,\beta, \gamma\in\oneforms(X)$ are holomorphic  1-forms. Since $\text{dim}_{\mathbb{C}}(\oneforms(X))=2$ we obtain a holomorphic map $\bbP^1\rightarrow \mathbb{P}(\oneforms(X))\simeq \bbP^1$ defined by $y\mapsto [\alpha y^2+\beta y+\gamma]$. It has at most degree two. If its degree is smaller than two, by using homogeneous coordinates in $\mathbb P^1$, we see that there is an invariant horizontal $X\times y$ for $\mathcal{F}$, which does not occur since our system $(X,A)$ is irreducible (see also \eqref{normalfA}). Hence the map is a degree two ramified covering. The fibers $p_i$ of the statement correspond to forms in $\oneforms(X)$ having a single zero. Recall that the hyperelliptic involution in $X$ fixes precisely six points and that for each of them there corresponds a unique element in $\mathbb{P}(\oneforms(X))$ having a single zero at the given fixed point of the involution. In fact these are the only one-forms in $\mathbb{P}(\oneforms(X))$ with a single zero. Hence after the degree two branched covering there are 12 fibers with a single zero.

\end{proof}

By using Lemma \ref{l:groethendieck_riemann_roch}, we get $$c_1(\wedge^2 T\mathcal{M}_{|\mathcal{R}})=4-\frac{12}{2}=-2.$$ By the argument at the beginning of the section, we have $2+d_2=-2$ and therefore $d_2=-4$.
\end{proof}

\section{Second proof of Theorem \ref{t:localdiffeo}} \label{s:proof}

Let $(X,A)\in\SYS^{irr}$. If $\text{Mon}$ is locally injective at $(X,A)$, then it is a non-constant holomorphic map between spaces of the same dimension, therefore open around $(X,A)$. { Local injectivity and openness imply that $\text{Mon}$ is a local biholomorphism.}

We will use the correspondence between systems in $\SYS^{irr}$ with monodromy $\rho$ and embedded rational curves in $\mathcal{M}_{2,\rho}$ proved in Section \ref{s:rational curves correspond to systems} to see that the case where $\text{Mon}$ is not locally injective at $(X,A)$ cannot occur.

Suppose, to reach a contradiction,  that for each neighbourhood $U_{\varepsilon}$ of $(X,A)$ there exist distinct $(X_1,A^{\varepsilon}_1)$ and $(X_2,A^{\varepsilon}_2)$ satisfying $\text{Mon}(X_1,A^{\varepsilon}_1)=\text{Mon}(X_2,A^{\varepsilon}_2)=\rho_{\varepsilon}$.  { The translation of this data under the dictionary between systems and rational curves in spaces of branched projective structures (see  Lemma \ref{l:rational curve}) produces the following geometric situation: an embedded rational curve  $\sigma_A$ in the leaf $\mathcal{M}_{2,\rho}$ of the foliated complex manifold $\mathcal{M}_2$ and for each $\varepsilon>0$, a pair of disjoint embedded rational cuves $\sigma_{A^{\varepsilon}_1},\sigma_{A^{\varepsilon}_2}$ in $\mathcal{M}_{2,\rho_{\varepsilon}}\subset \mathcal{M}_2$ that lie in the tubular $\varepsilon$-neighbourhood $V_{\varepsilon}$ of the  rational curve $\sigma_A$ in the foliated complex manifold $\mathcal{M}_2$ (we use Theorem \ref{t:isomonodromy foliation} here). This implies that the distance between the rational curves $\sigma_{A^{\varepsilon}_1},\sigma_{A^{\varepsilon}_2}$ tends to zero with $\varepsilon$.
On the other hand, the foliated tubular structure is quite simple around $\sigma_A$. Indeed, since $\sigma_A$ is simply connected, the foliation around $\sigma_A$ is holomorphically a product $U_A\times\mathbb{D}^{6}$ where $U_A$ is a neighbourhood of $\sigma_A$ in $\mathcal{M}_{2,\rho}$. Using  Proposition \ref{p:negative selfintersection}, we can find a uniform $\delta>0$  so that in the $\delta$-neighbourhood of $\sigma_A\times z$ in $U_A\times z$  the only rational curve is $\sigma_A\times z$. This contradiction shows that $\text{Mon}$ is locally injective.

\section{Relation between the two proofs of Theorem \ref{t:localdiffeo}}\label{Sec:Comparaison}

The two proofs of Theorem \ref{t:localdiffeo} given so far can be ``compared'' in the moduli space $\CON$ of $\mathfrak{sl}_2$-connections on rank two holomorphic bundles over genus two curves.
Projectivization of the fibres of the bundles subjacent to the elements in $\CON$ gives a natural map
$$\mathbb P\ :\ \CON\rightarrow \PCON$$
from $\CON$ to the moduli space of flat $\mathbb{P}^1$-bundles over genus two curves.
Remark that each fibre of this compactification map contains at most one system. Indeed, suppose $f,g:\widetilde{C}\rightarrow\text{SL}_2(\mathbb{C})$ are equivariant maps (solutions of two systems) with respect to representations $\rho_f, \rho_g$ satisfying $\rho_g=m\cdot \rho_f$ for some homomorphism  $m:\pi_1(C)\rightarrow\{\pm\text{Id}\}$. The product $h:=f^{-1} g$ becomes $m$-equivariant.  By projectivising, $[h]$ defines a holomorphic map $C\rightarrow \text{PSL}_2(\mathbb{C})$, hence constant. We deduce that $h$ is constant and $m$ trivial. Moreover, we want to emphasize that the map $\CON\rightarrow \PCON$, which is naturally
a $16$-fold ramified cover (whose Galois group is the group of $2$-torsion points of $\mathrm{Jac}(X)$,
see \cite[Section 2.1.1]{HL}), is actually \'etale near $\SYS$. In order to understand isomonodromic deformations near
$\SYS$, it is equivalent to work in $\CON$ or in $\PCON$.

Elements of $\PCON$ are better viewed as equivalence classes of pairs $(P,\mathcal F)$ where $P\to X$ is a ruled surface
and $\mathcal F$ a Riccati foliation on $P$ (i.e. transverse to the ruling). A projective structure with $2g-2$ branch points
is defined by an element $(P,\mathcal F)$ of $\PCON$ together with the choice of a section $\sigma:X\to P$ having
self-intersection $0$ (see formula (\ref{eq:number of tangencies})). Thinking of $(P,\mathcal F)$ as the projectivization
of a $\mathfrak{sl}_2$-connection $(E,\nabla)$, such a section $\sigma$ corresponds to a line sub-bundle $L\subset E$
having degree $0$, implying therefore that the vector bundle $E$ is strictly semi-stable, i.e. semi-stable but not stable (see \cite{NaRam}).
Consider the subspace $\CONSSS\subset\CON$ of connections over strictly semi-stable bundles.
It is a singular hypersurface that contains $\SYS$ as a proper subset.
As we will see, $\SYS$ is in the singular set of $\CONSSS$. We claim that around a generic point of $\SYS$,
$\CONSSS$ is locally a product $V\times B$ where $V$ is a surface singularity of type $A_1$, $B$ is a ball of dimension $6$ and the isomonodromy foliation is given by the projection $V\times B\rightarrow B$. The points of $\SYS$ correspond to the singular locus of this product.

For notational symplicity, we will also denote the image of $\CONSSS$
and $\SYS$ in $\PCON$ as $\PCONSSS$ and $\PSYS$ respectively.
We construct a forgetful map
\begin{equation}\label{desing}
 \mathcal{M}_2\twoheadrightarrow\PCONSSS\subset \PCON
\end{equation}
that associates to each branched projective structure $(P,\mathcal F,\sigma)$, the class in the moduli space
$\PCON$ of the Riccati foliation $(P,\mathcal F)$. This makes sense at least locally, over a neighborhood
of $\PSYS$, or also globally, with  respect to the natural map $\Teich\to\Mod$ for the underlying curve.

By construction, leaves of the isomonodromic foliation on  $\mathcal{M}_2$ are sent to leaves of $\PCONSSS$
and any rational curve in $\mathcal{M}_2$ is contracted to some point in $\PSYS$ by the map (\ref{desing}).
We claim that this map is a desingularization map around generic points of $\PSYS$.  The next lemma is the key to prove this claim.

\begin{lemma}\label{lem:hyperelliptic involution on M2}
  The action of the hyperelliptic involution $h:X\to X$ on $\mathcal{M}_2$
  \begin{itemize}
    \item[(i)]  preserves each leaf of the isomonodromy foliation,
  \item[(ii)]  preserves each rational curve in a leaf fixing at most two points,
  \item[(iii)] does not have isolated fixed points.
  \end{itemize}
\end{lemma}

\begin{proof}
(i) By \cite[Theorem 10.2]{Goldman2}, the action of the hyperelliptic involution is trivial
on the $\text{SL}_2(\mathbb{C})$-character variety on a genus two surface.

(ii) First remark that the hyperelliptic involution acts trivially on the Teichm\"uller space $\Teich$. Hence each fibre of the forgetful map $\mathcal{M}_{2,\rho}\rightarrow \Teich$ is preserved by the action of the hyperelliptic involution. By Lemma \ref{l:fibers over Teich} the fibers contain either one point, two points or a rational curve.

 Next we prove that the rational curves cannot be fixed pointwise.  Let $\mathcal{D}:\widetilde{X}\times\mathbb{P}^1\rightarrow \mathbb{P}^1$ be the holomorphic family of developing maps describing a rational curve in some $\mathcal{M}_{2,\rho}$ and $h:X\rightarrow X$ denote the hyperelliptic involution. Suppose that there are more than two fixed points of $h$ on the rational curve.  Then every point is fixed, and  the holomorphic family of developing maps $\widehat{\mathcal{D}}(x,p):=\mathcal{D}(\widetilde{h}(x), p)$ defines the same rational curve in $\mathcal{M}_{2,\rho}$. For each $p\in\mathbb{P}^1$, the maps $\mathcal{D}(\cdot,p)$ and $\widehat{\mathcal{D}}(\cdot, p)$ define the same point in $\mathcal{M}_{2,\rho}$ and there exists $\phi_p\in\text{PSL}(2,\mathbb{C})$ such that $$\hat{\mathcal{D}}(x,p)=\phi_p\circ \mathcal{D}.$$ By construction the map $\mathbb{P}^1\rightarrow \text{PSL}(2,\mathbb{C})$ defined by  $p\mapsto\phi_p$ is holomorphic, hence constant $\phi_p=\phi$.
Since $\rho$ is irreducible, the action of $h$ can be interpreted from $X\times\mathbb{P}^1$ to $X\times\mathbb{P}^1$, preserving each horizontal and the Riccati foliation. Hence a form of type
$$\mathrm{d}z=\alpha z^2+\beta z+\gamma$$
where $\alpha,\beta$ and $\gamma$ belong to $\Omega^1(X)$ defines the Riccati foliation.
The invariance under the hyperelliptic involution $h:X\to X$ of the Riccati foliation implies that the one-forms $\alpha,\beta$ and $\gamma$ are each invariant under the hyperelliptic involution. Hence each defines a holomorphic one-form on $X/h\equiv \mathbb{P}^1$. Hence $\alpha=\beta=\gamma=0$, which is impossible. Hence there are at most two fixed points of the action of $h$ on each rational curve in $\mathcal{M}_2$.
In fact, the normal form (\ref{normalfA}) for $\text{SL}_2(\mathbb{C})$-systems induces after projectivization
a Riccati equation of the form $\mathrm{d}z=\alpha z^2+\gamma$ which are invariant by $(x,y,z)\mapsto(x,-y,-z)$
(recall $h(x,y)=(x,-y)$) and only $\sigma(x,y)=(x,y,0)$ or $(x,y,\infty)$ are invariant.

(iii) As in the previous proof, a fixed point $\sigma$ for the action of the hyperelliptic involution on $\mathcal{M}_{2}$ corresponds to the pull back of a singular projective structure $\sigma_h$ on $X/h\equiv \mathbb{P}^1$ by the projection map $X\rightarrow X/h$. The angle of $\sigma_h$ at a point $[x]\in X/h$ is the angle of $\sigma$ at $x$ if $x$ is not fixed by $h$ and half the angle of $\sigma$ at $x$ otherwise.
On the other hand, branch points are invariant under the action of $h$.
If no branch point of $\sigma$ is fixed by $h$, then $\sigma_h$ has a simple branch point of angle $4\pi$ and six points of angle $\pi$. As we have mentioned in section \ref{s:the tangent bundle}, we can move the branch points. This operation can be done by a surgery called Schiffer variations (see \cite[p.391]{CDF}) that deform $\sigma_h$ isomonodromically and continuously in the space of singular projective structures on $\mathbb{P}^1$. The double covers of $\mathbb{P}^1$ branching at the points of angle $\pi$ of this isomonodromic deformation produce a nontrivial isomonodromic deformation in $\mathcal{M}_{2}$ that is fixed by the hyperelliptic involution.
If some branch point is fixed by $h$, then $\sigma$ has a single branch point of angle $6\pi$. Therefore $\sigma_h$ has a singularity of angle $3\pi$. Since the angle is bigger than $2\pi$ we can find two paths that have the same image by the developing map. Applying a Schiffer variation along this pair of twins we obtain a singular projective structure with 6 points of angle $\pi$ and a point of angle $4\pi$.
 Again,  the pull-back of the deformation by the hyperelliptic branched double cover produces a path of isomonodromic projective structures in $\mathcal{M}_2$ that are fixed by the hyperbolic involution.
\end{proof}

 As a consequence, we have that the quotient of each leaf $\mathcal{M}_{2,\rho}/h$ is a regular complex surface and each smooth $(-4)$-rational curve $\mathcal{R}\subset\mathcal{M}_{2,\rho}$  produces a smooth $(-2)$-rational curve $\mathcal{R}/h\subset \mathcal{M}_{2,\rho}/h$. The restriction of the map induced by (\ref{desing}) to a neighbourhood of the curve in $\mathcal{M}_{2,\rho}/h$ is precisely the collapse of the $(-2)$-curve and gives a surface singularity of type $A_1$. For generic $\rho$ the space $\mathcal{M}_{2,\rho}/h$ embeds in the moduli space of branched projective structures.

In fact, regarding $0$-sections (i.e. sections having self-intersection $0$) of a $\mathbb P^1$-bundle over a genus $2$ curve,
we have the following (see \cite[Theorem 2.1 and Section 3]{HL})

\begin{lemma}\label{lem:classificationbundles}
Given a branched projective structure $(P,\mathcal F,\sigma)$ in $ \mathcal{M}_2$, we are in one of the following cases:
\begin{itemize}
\item {\bf generic bundles}: $P$ is decomposable but not trivial, i.e. admitting exactly two disjoint $0$-sections $\sigma,\sigma'$,
and they are permuted by $h$,
\item {\bf unipotent bundles}: $P$ is undecomposable, $\sigma$ is the unique $0$-section
and it is invariant by $h$,
\item {\bf trivial bundle}: $P=X\times\mathbb P^1$, the set of $0$-sections is isomorphic to $\mathbb P^1$,
and $h$ acts as a Moebius involution on it, fixing exactly two of them.
\end{itemize}
\end{lemma}

Going back to the notations of Section \ref{Sec:FuchsianSystems}, we have that
$$(t_1,t_2,t_3,z_1,z_2,z_3,c_1,c_2,c_3)\in\mathbb C^6$$
is an open chart of $\CON$ or $\PCON$ containing a Zariski open subset of $\SYS$ or $\PSYS$.
Precisely, $(t_1,t_2,t_3)$ are local coordinates on $X\in\Teich$ or $\Mod$, $(z_1,z_2,z_3)$ parametrize bundles
over $X$, and $(c_1,c_2,c_3)$ stand for those connections on a given bundle.
The locus of the trivial bundle is given by
$$\SYS=\{z_1=z_2=z_3\ \text{and}\ c_1+c_2+c_3=0\},$$
or equivalently $\SYS=\{Z_1=Z_2=Z_3=0\}$ where
$$Z_1=c_1+c_2+c_3,\ \ \ Z_2=z_2-z_1\ \ \ \text{and}\ \ \ Z_3=z_3-z_1.$$
The locus of unipotent bundles has codimension $2$, given by
$$\mathcal U_1:=\{z_1=z_2=z_3\}\ \ \ \cup\ \ \ \mathcal U_2:=\{c_1+c_2+c_3=c_1z_1+c_2z_2+c_3z_3=0\},$$
or equivalently $\mathcal U_1=\{Z_2=Z_3=0\}$ and $\mathcal U_2=\{Z_1=c_2Z_2+(c_2+c_3)Z_3=0\}$.
The locus of strictly semi-stable bundles forms a hypersurface $\Upsilon\subset\CON$
that contains unipotent and trivial bundles.
As explained above, $\Upsilon$ has the structure of $A_1\times B$
near a point of $\SYS$. More precisely, we find that the tangent cone of $\Upsilon$ along $\SYS$ is given by
$$\left( 2 t_2 z_1- t_2- z_1 \right)  \left(  t_1- t_3 \right)  Z_1 Z_2
- \left( 2 t_3 z_1- t_3- z_1 \right)  \left( t_1- t_2 \right)  Z_1 Z_3$$
\begin{equation}\label{eq:tangentconic}
+\left( - c1 \left( 2 t_2-1 \right)  \left(  t_1- t_3 \right)
- c_2\left( 2 t_1 t_2+2 t_3 t_1-4 t_3 t_2-2 t_1+ t_2+ t_3 \right)  \right) Z_2 Z_3
\end{equation}
$$+ c_2 \left( 2 t_2-1 \right)  \left(  t_1- t_3 \right)  Z_2^{2}+  \left( t_3-1 \right)  \left(  t_1-t_2 \right)  \left(  c_1+c_2 \right)  Z_3^{2}=0
$$
This computation is similar to \cite[Proposition 5.2]{HL}: we compute the classifying map $\CON\to\mathbb P^3_{NR}$
towards the Narasimhan-Ramanan moduli space of bundles (see \cite{NaRam}) and then pull-back the Kummer
equation of strictly semi-stable bundles.
As we have already shown, a generic isomonodromy leaf $\mathcal L$ in $\CON$ is transversal to $\SYS$:
they intersect as a point $\mathcal L\cap\SYS=\{Q\}$. In other words, the point $Q$ is defined by fixing parameters
$(t_1,t_2,t_3,c_1,c_2,z_1) \in \mathbb C^6$ and the corresponding conic (\ref{eq:tangentconic}) is the tangent cone
of $\mathcal L_0:=\mathcal L\cap\Upsilon$.  A straightforward computation shows that this conic is smooth if, and only if,
the system corresponding to $Q$ is irreducible, i.e. (\ref{RedLocus});
in this case, $\mathcal L_0$ is a surface having a singular point of type $A_1$ at $Q$.
Moreover, $\mathcal L_0$ intersects each component of the unipotent locus $\mathcal U_1\cup\mathcal U_2$
as a smooth curve.
After blowing up $\SYS$, the strict transforms $\widetilde{\mathcal L_0}$ becomes non singular
near the exceptional divisor $\widetilde\SYS$, the conic $\widetilde{\mathcal L_0}\cap\widetilde\SYS$
has self-intersection $-2$ in $\widetilde{\mathcal L_0}$ and the two curves $\widetilde{\mathcal L_0}\cap\widetilde{\mathcal U_i}$
are smooth and disjoint curves: they are the ramification locus of the $2$-fold cover
$\mathcal{M}_{2,\rho}\to\mathcal L_0:=\mathcal{M}_{2,\rho}/h$ considered above.
Details of the above computations will appear in another paper.

}

\section{The higher genus case}\label{r:counter-example}

Theorem \ref{t:localdiffeo} fails in higher genus. An instance of this phenomenon is given by
pulling back a system $(X,A)$ on a genus two Riemann surface  with irreducible monodromy by a parametrized family of ramified coverings $f_t:S_t\rightarrow X$ of fixed degree. For instance, if
we consider for $f_t$ the (irreducible) family of degree two covers $f_t:S_t\stackrel{2:1}{\rightarrow} X$
ramifying over two points, then $S_t$ is a deformation of genus $4$ curves and the deformation
of the pair $f_t^*(X,A)$ is obviously isomonodromic (i.e. with constant monodromy) and irreducible.
We can construct similar deformations with $S_t$ having arbitrary genus $g\ge4$.
{ For genus $g=3$ we need different techniques, since Hurwitz formula does not allow $S_t$ to have genus 3 in this procedure }

In general, infinitesimal rigidity of a system $(X,A)$ can be detected on its eigencurve $C\subset X \times \mathbb P^1$, defined as the projectivization of the eigenvectors of $A$. More precisely, if we denote by $\Pi_2 : X\times \mathbb P^1 \rightarrow \mathbb P^1$ the horizontal fibration, the system $(X,A)$ is rigid if and only if every section of  $N_\mathcal F\simeq (\Pi _2)_{|C}^* T \mathbb P^1$ is the pull-back of a section of $T\mathbb P^1$. This is due to Lemma \ref{l:tangent bundle}, together with an infinitesimal version of Corollary \ref{l:rational curves}. This remark allows to construct infinitesimal deformations of other kinds of systems.

The first ones are the systems  $(X,A)$ whose eigencurve $C\subset X\times \mathbb P^1$ has a vertical component and no horizontal one.  They are generalizations of those coming from ramified covering techniques.  In this case,  the vertical component of $C$ intersects the union of the other components in two points $p$ and $q$ (that might coincide): hence the section of $\Pi_2^* (T\mathbb P^1)_{|C}$ defined by a non zero holomorphic vector field on the vertical component that vanishes at $p$ and $q$, and  by zero on the other components, is not the pull-back of a section of $T\mathbb P^1$ by $\Pi_2$. In particular, the system is not infinitesimally rigid.  As we proved, this is not possible in genus two to find such a system unless there is a component of the eigencurve that is horizontal (which is equivalent to reducibility). However in higher genus it is always possible. { In genus $3$, infinitesimally non rigid examples can be provided by constructing explicit systems using the information on the eigencurve}. Indeed, let $X$ be any smooth curve of genus $\geq 3$, and $x \in X$. The hyperplane $H$ of $\Omega^1(X)$ consisting of forms vanishing at $x$ has dimension $\geq 2$. A generic degree two map $\mathbb P^1$ in $\mathbb P (H)$ comes, as in the proof of Lemma \ref{l:nongeneric horizontal fibers}, from a system whose corresponding eigencurve curve contains the vertical $x\times \mathbb P^1$, but no horizontal.

{ Here is another family of examples where the Riemann-Hilbert mapping is not an immersion, but the eigencurve does not necessarily contain a vertical fibre.} Assume that the eigencurve $C$ is symmetric with respect to the involution of $X\times \mathbb P^1$ given by $(x,y) \mapsto (x , -y)$, which happens if the coefficient $\alpha$ of the system vanishes identically. Notice that when $\beta$ and  $\gamma$ are not $\mathbb C$-proportional, then the system is irreducible, and that if $\beta$ and $\gamma$ do not share a common zero, then the curve $C$ is smooth. We claim that in this case, we have a subspace of infinitesimal isomonodromic deformations for our system whose dimension is $\geq g-3$. From Lemma \ref{l:tangent bundle}, we know that this space has dimension
$$h^0 (\Pi_2 ^* T\mathbb P^1, C) - h^0 (T\mathbb P^1, \mathbb P^1)  = h^0 (\Pi_2 ^*  \OO(2) , C) - 3 .$$
Denote   $ \pi_2: X \rightarrow \mathbb P^1$ the meromorphic function  $\pi_2 = \beta / \gamma$, and by $r: \mathbb P^1 \rightarrow \mathbb P^1$ the double covering $r(y)= y^2$. We have the relation $$ \pi_2 \circ  \Pi_1 = r \circ \Pi_2 .$$
Since the line bundle $\OO(2)$ over $\mathbb P^1$ is the preimage by $r$ of $\OO(1)$, and that the preimage of $\OO(1)$ by $\pi_2$ is $K_X$, we get
$$ \Pi_2 ^* \OO (2) = \Pi_1^* ( \pi_2 ^* \OO(1) ) = \Pi_1^* ( K_X ) .$$
We thus obtain
$$h^0 (\Pi_2 ^* \OO(2) , C) \geq h^0 (K_X, X) = g, $$
which shows our claim. The precise condition ensuring the Riemann-Hilbert mapping to be an immersion in terms of properties of the eigencurve $C$ seems hard to find in general.
\section{Appendix: the complex structure on $\mathcal{M}_k$}
\label{s:appendix}
In this appendix we prove Theorem \ref{t:isomonodromy foliation}.

Suppose $S$ is an orientable compact connected topological surface of genus $g\geq 2$, $k\geq 0$ an integer and $\widetilde{S}\rightarrow S$ a fixed universal covering map.

As with $\mathcal{M}_{k,\rho}$, the complex structure on $\mathcal M_{k}$ is the complex structure defined by the deformation theory of projective structures. That is, assume that $Y$ is a complex analytic space and that we have the following data :
\begin{enumerate}
\item[a)]  a holomorphic submersion $\Pi : Z\rightarrow Y$ having as fibers simply connected Riemann surfaces,
\item[b)] a free proper discontinuous action of $\Gamma = \pi_1(S)$ on $Z$ preserving fibers of $\Pi$,
\item[c)] a holomorphic map $\mathcal D : Z\rightarrow \mathbb P^1$  whose restriction to the fiber over each $y\in Y$ has a $k$ critical orbits counted with multiplicity and is equivariant with respect to a (family of) representation $\rho(y):\Gamma\rightarrow \text{PSL}_2(\mathbb{C})$.
\end{enumerate}
Hence the restriction of $\mathcal{D}$ to the fiber over $y\in Y$ defines a branched projective structure over $S$ that we call $\sigma(y)$.
The complex structure that we are going to construct on $\mathcal{M}_k$ is such that the natural map $Y\rightarrow \mathcal{M}_k$ defined by $y\mapsto \sigma(y)$ is holomorphic.

Let $\sigma_0\in\mathcal{M}_k$ be a point whose equivariance is $[\rho_0]$. Take a small neighbourhood $V\subset S$ formed by a union of disjoint round discs $V_i$  (for $\sigma_0$), each containing one point of  $\text{div}(\sigma_0)$ and  $U$ neighbourhood of $[\rho_0]$ in the character variety. Let $\mathcal{H}_V=\Pi_{i}\mathcal{H}_{V_i}$ denote the product of Hurwitz spaces associated to each disc in $V$.  We define a chart around $\sigma_0$ by using the following

\begin{lemma}\label{l:chart}
  There exists a holomorphic deformation $\Pi: Z\rightarrow U\times\mathcal{H}_V$ of branched projective structures such that the map $U\times\mathcal{H}_V\rightarrow\mathcal{M}_k$ is injective around the base point.
\end{lemma}

In $\mathcal{M}_k$ we define the topology generated by images of open sets  via the maps produced in Lemma \ref{l:chart}. The induced topology on each $\mathcal{M}_{k,\rho}\subset\mathcal{M}_k$ coincides with the cut-and-paste topology defined in \cite{CDF}.

\begin{lemma}\label{l:separation}
The topological space $\mathcal{M}_k$ is separated and the maps in Lemma \ref{l:chart} provide a topological atlas on $\mathcal{M}_k$.
\end{lemma}

To check that the transition maps are holomorphic with respect to the natural product complex structures on the $U\times\mathcal{H}_V$'s it suffices to check that for every holomorphic deformation $Y\rightarrow\mathcal{M}_k$ whose image is contained in a chart, the natural projections onto $U$ and $\mathcal{H}_{V_i}$ are holomorphic. The projection onto $U$ is trivially holomorphic by definition of a holomorphic deformation.  The holomorphicity of the other projections needs only to be checked around the points in $Y$ corresponding to the generic stratum of the deformation, since by continuity and Riemann Extension Theorem, the map will be also holomorphic  at the points on non-generic strata.
We are therefore reduced to proving the following

\begin{lemma}\label{l:holomorphic maps}
  Let $\Pi:Z\rightarrow (Y,y_0)$ be a germ of holomorphic deformation of branched projective structures all lying in a fixed stratum. Then, in any chart $U\times\mathcal{H}_V$ of  $\mathcal{M}_k$ containing the point associated to $y_0$, the projection $(Y,y_0)\rightarrow \mathcal{H}_{V_i}$ is holomorphic.
\end{lemma}

 To finish the proof of Theorem \ref{t:isomonodromy foliation} we just need to remark that the monodromy map  is given in any chart by a projection onto the $U$-factor, thus it is a local holomorphic submersion. Let us now prove the preceeding lemmas.

 \begin{proof}[Proof of Lemma \ref{l:chart}]
 Over each $\mathcal{H}_{V_i}$ there exists a holomorphic family $Z_i\rightarrow \mathcal{H}_i$ of branched projective structures on the disc $V_i$ with prescribed boundary values that contains the restriction of $\sigma_0$ to $V_i$ (see \cite{CDF}).

 Choose a  lift of the chosen neighbourhood $U$ of $[\rho_0]$ to $\text{Hom}(\pi_1(S),\text{PSL}_2(\mathbb{C}))$ around $\rho_0$  by imposing that no two points belong to the same $\text{PSL}_2(\mathbb{C})$-orbit  and, by abuse of language,  call it $U$. We will first produce a holomorphic deformation $X_U \rightarrow U$ of branched projective structures over $S$,  such that the monodromy of the branched projective structure over a point $\rho\in U$ is the homomorphism $\rho$.
We can moreover assume the existence of an open set $W\subset X_U$
holomorphically equivalent to $U \times V$ such that the holomorphic map $X_U\setminus W\rightarrow U$ is equipped with a deformation of branched projective structures on $S\setminus V$ with prescribed boundary values. The latter can be glued to the $Z_i$'s to form a holomorphic deformation $Z\rightarrow U\times \mathcal{H}_V$ of $\sigma_0$, as desired.

 As a smooth manifold, $X_U$ is defined to be the product $U\times S$. Remark that $\sigma_0$ induces a natural complex structure $\tau_0$ on $S$. The group $\pi_1(S)$ acts on $U\times \widetilde{S}\times \mathbb{P}^1$ by $\gamma\cdot(\rho,z,w)=(\rho, \gamma\cdot z, \rho(\gamma)(w))$. The action preserves the fibrations onto the first two factors, and the 'horizontal' smooth foliation by real surfaces. On the quotient space $W_U$, the $\mathbb{P}^1$-fibration $\pi_U: W_U\rightarrow U\times S$ is well defined and the induced foliation $\mathcal{F}$ is transversely holomorphic. The complex structure $\tau_0$ on $S$ defines a complex structure on $W_U$ and on $U\times S$. With respect to this complex structure, the fibration $\pi_U$ together with the foliation $\mathcal{F}$ are locally holomorphically trivializable. Consider a covering $\{W_i\}$ of $\rho_0\times S\subset U\times S$ where the trivialization is possible and such that some neighbourhood of each $V_i$ is compactly contained in some $W_i$. Over each $W_i$ the section $\sigma_0$ can be extended to a holomorphic section defined on $W_i$ (with respect to the given holomorphic structure) by imposing $\sigma(\rho,z)=\sigma(\rho_0,z)$. At the intersections $W_i\cap W_j$, the chosen continuations do not coincide,  but they are transverse to the foliation $\mathcal{F}$. By using a partition of unity associated to the $W_i$'s, we can change the sections to a smooth section $\sigma:X_U\rightarrow W_U$ that is still transverse to the foliation on the $W_i\cap W_j$'s. On the neighbourhoods of the $V_i$'s, the section coincides with the first (holomorphic) extension. By construction, the tangencies  between the foliation $\mathcal{F}$ and the section $\sigma$ correspond by projection to the set $T=\sqcup U\times q_i$ where $q_i$ is a point of $\text{div}(\sigma_0)$. The transversely holomorphic structure of $\mathcal{F}$ induces a holomorphic structure on $X_U\setminus T$. Around $T$, we also have a holomorphic structure defined by the initial complex structure. By construction these complex structures are compatible, so they produce a complex structure on $X_U$. This complex structure induces a complex structure on $W_U$ such that the section $\sigma$, the foliation $\mathcal{F}$ and the projection $X_U\rightarrow U$ are holomorphic. The analytic continuation of any germ of $\sigma$ at a point in $p_0\times S$  produces a holomorphic map $\widetilde{X_U}\rightarrow \mathbb{P}^1$ defined on the universal cover of $X_U$ with all the properties of a holomorphic deformation of $\sigma_0$. Over a point $\rho\in U$, the monodromy of the associated branched projective structure is the homomorphism $\rho$. The local holomorphic triviality of the fibration $\pi_U$ equipped with the foliation $\mathcal{F}$ and the section $\sigma$ around a $\rho_0\times V_i$ allows to glue the given model to the holomorphic deformations $Z_i\rightarrow \mathcal{H}_{V_i}$ as a complex manifold and construct the desired holomorphic deformation $Z\rightarrow U\times\mathcal{H}_V$ of $\sigma_0$.

 To prove the injectivity of the associated map $c: U\times\mathcal{H}_V\rightarrow \mathcal{M}_k$, remark that for each point in $\mathcal{M}_k$ there is a well defined class of monodromy homomorphism. Since by construction no pair of points in $U$ are conjugated by a non-trivial element in $\text{PSL}_2(\mathbb{C})$, we have that $c(\rho_1,h_1)=c(\rho_2,h_2)$ implies $\rho_1=\rho_2$. On the other hand, we know by \cite{CDF} that the restriction $h\mapsto c(\rho_1, h)$ is injective; hence $h_1=h_2$.
\end{proof}

\begin{proof}[Proof of Lemma \ref{l:separation}] Suppose $\sigma_0$ and $\sigma_1$ are two points in $\mathcal{M}_k$ that are not separated. If their monodromies are not conjugated, we can find disjoint open neighbourhoods $U_0$ and $U_1$ around them and construct a chart using these open sets having disjoint images. Hence both can be thought of as smooth sections of a flat bundle $S\times_{\rho_0}\mathbb{P}^1$. If the complex structures  on $S$ (and therefore on the flat bundle) induced by $\sigma_0$ and $\sigma_1$ are not equivalent, then we can choose small open sets $U_0, V_0$ and $U_1,V_1$ that define charts and such that the projection of the image of the associated chart to Teichm\" uller space is contained in disjoint sets. Hence we can suppose that $\sigma_0$ and $\sigma_1$ induce the same complex structure $X$ on $S$.
 Let $z\in\mathbb{D}\simeq \widetilde{X}$ be a uniformizing variable of $X$. If $\mathcal{D}_i$ is the developing map with equivariance $\rho_0$ associated to $\sigma_i$, the meromorphic quadratic differential $$\{\mathcal{D}_i(z),z\}\mathrm{d}z^2\quad\text{ on }\widetilde{X}$$ where $\{f(z),z\}$ denotes the Schwarzian derivative of $f$ with respect to $z$, descends to $X$. Its poles lie precisely on the branched points of $\sigma_i$. Denote by  $q_0$ and $q_1$ respectively the meromorphic quadratic differentials  on $X$ thus obtained for $i=0,1$. The branched projective structures $\sigma_0$ and $\sigma_1$ are equivalent if and only if $q_0=q_1$ on $X$. However, this last equality follows as soon as we find an open set of $X$ where it is satisfied.  Suppose that there exists a sequence $\sigma_n\in\mathcal{M}_k$ converging to $\sigma_0$ and to $\sigma_1$. Take a point $z_0\in X$ where neither $q_0$ nor $q_1$ has a pole. In a neighbourhood of $z_0$ in the universal curve bundle over $\mathcal{M}_k$ we can extend the germ at $z_0$ of the coordinate $z$ on $X$ holomorphically as a local holomorphic coordinate on each fibre. Let $\rho_n$ be represnetatives of the equivariance of $\sigma_n$ such that $\rho_n\rightarrow \rho_0$. Denote by $\mathcal{D}_n$ the developing map of $\sigma_n$ with equivariance $\rho_n$. For each $n$ the  map $q_n(z):=\{\mathcal{D}_n(z),z\}$  is holomorphic and injective on the intersection of the neighbourhood of $z_0$ with the fibre of the universal curve bundle corresponding. As $n$ tends to infinity the sequence of germs $q_n$ has a unique holomorphic limit germ of function. By construction this limit coincides with both $q_0$ and $q_1$. Therefore $q_0=q_1$ around $z_0$ and $\mathcal{M}_k$ is separated.
 By definition the maps defined in Lemma \ref{l:chart} are local homeomorphisms.
\end{proof}

\begin{proof}[Proof of Lemma \ref{l:holomorphic maps}]
The key remark for the proof is that the complex structure on the stratum $B$ of the base point of $\mathcal{H}_{V_i}$ has very special holomorphic charts. In fact, the map that sends, to every $b\in B$, the image of the branch point by the developing map of the associated holomorphic deformation, is a local holomorphic diffeomorphism (see \cite{CDF}).
Thus it suffices to check that for an arbitrary holomorphic deformation $Z\rightarrow (Y,y_0)$ lying in a fixed stratum the image of the critical point by the developing map defines a holomorphic map. This is obvious by the definition of a holomorphic deformation of branched projective structures.
\end{proof}

\end{document}